\documentclass[11pt,a4paper]{article}

\usepackage[english]{babel}
\usepackage[utf8]{inputenc}
\usepackage{graphicx}
\usepackage{amsmath,amsthm,amssymb}
\usepackage{hyperref}
\usepackage{mathtools}
\usepackage{comment}
\usepackage{aligned-overset}
\usepackage{enumitem}
\usepackage{caption}
\usepackage[a4paper,top=3cm,bottom=3cm,left=3cm,right=3cm,marginparwidth=1.75cm]{geometry}
\usepackage{xcolor}
\usepackage[colorinlistoftodos]{todonotes}

\newcounter{theoremintro}

\newtheorem{theorem}{Theorem}
\newtheorem{theoremi}[theoremintro]{Theorem}

\newtheorem{prop}[theorem]{Proposition}
\newtheorem{lemma}[theorem]{Lemma}
\newtheorem{corollary}[theorem]{Corollary}
\newtheorem{corollaryi}[theoremintro]{Corollary}

\theoremstyle{definition}
\newtheorem{definition}[theorem]{Definition}
\newtheorem{definitioni}[theoremintro]{Definition}
\newtheorem{notation}[theorem]{Notation}
\theoremstyle{remark}

\newtheorem{remark}[theorem]{Remark}
\setlist{nosep}
\setenumerate[1]{label=(\arabic*)}
\numberwithin{equation}{section}
\numberwithin{theorem}{section}

\newcommand{\overbar}[1]{\mkern 1.5mu\overline{\mkern-1.5mu#1\mkern-1.5mu}\mkern 1.5mu}
\newcommand{\R}{\mathbb{R}}
\newcommand{\N}{\mathbb{N}}
\newcommand{\Q}{\mathbb{Q}}
\newcommand{\Z}{\mathbb{Z}}
\newcommand{\C}{\mathbb{C}}
\renewcommand{\tilde}{\widetilde}

\newcommand{\acts}{\curvearrowright}
\newcommand{\e}{\varepsilon}
\newcommand{\ssubset}{\subset\joinrel\subset}
\newcommand{\cstar}{$\mathrm{C}^*$}

\title{Equivariant property Gamma and the tracial local-to-global principle for \cstar-dynamics}
\author{\begin{tabular}{ccc} Gábor Szabó\thanks{Supported by BOF project C14/19/088 funded by the research council of KU Leuven, and research project G085020N funded by the Research Foundation Flanders (FWO).} & \& & Lise Wouters\thanks{Supported by PhD-grant 11B6620N funded by the Research Foundation Flanders (FWO).} \end{tabular} \medskip\\
KU Leuven, Department of Mathematics, \\
Celestijnenlaan 200B, 3001 Leuven, Belgium \medskip\\ 
\begin{tabular}{ccc} \texttt{gabor.szabo@kuleuven.be} & & \texttt{lise.wouters@kuleuven.be} \end{tabular}}
\date{\vspace{-5ex}}

\begin{document}
\maketitle

\begin{abstract}
\noindent We consider the notion of equivariant uniform property Gamma for actions of countable discrete groups on \cstar-algebras that admit traces.
In case the group is amenable and the \cstar-algebra has a compact tracial state space, we prove that this property implies a kind of tracial local-to-global principle for the \cstar-dynamical system, generalizing a recent such principle for \cstar-algebras exhibited in work of Castillejos et al.
For actions on simple nuclear $\mathcal{Z}$-stable \cstar-algebras, we use this to prove that equivariant uniform property Gamma is equivalent to equivariant $\mathcal{Z}$-stability, generalizing a result of Gardella--Hirshberg--Vaccaro.
\end{abstract}

\renewcommand{\baselinestretch}{0.5}\normalsize
\tableofcontents
\renewcommand{\baselinestretch}{1.00}\normalsize

\addcontentsline{toc}{section}{Introduction}
\section*{Introduction}

This article aims to extend the fine structure theory for actions of amenable groups on finite simple \cstar-algebras, in particular those covered by the Elliott program.
The classification of such \cstar-algebras, which mirrors the celebrated Connes--Haagerup classification of injective factors \cite{Connes76, Haagerup87}, has been nearly completed as a culmination of numerous articles by many researchers over the past decade, such as \cite{ElliottGongLinNiu15, TikuisisWhiteWinter17, GongLinNiu20, GongLinNiu20-2, ElliottGongLinNiu20, GongLin20, GongLin22, Schafhauser20}.
Furthermore, Carrion--Gabe--Schafhauser--Tikuisis--White have announced an eagerly awaited new conceptual proof of the classification theorem \cite{CGSTW}, which does not directly rely on the prior works related to tracial approximation.
By now it has been recognized that the next natural step is to understand the underlying symmetries of classifiable \cstar-algebras, which can be interpreted as the goal to classify group actions on them.
This mirrors the work of Connes, Jones, Ocneanu and others in \cite{Connes77, Jones80, Ocneanu85, SutherlandTakesaki89, KawahigashiSutherlandTakesaki92, KatayamaSutherlandTakesaki98, Masuda07, Masuda13}.
When it comes to classifying group actions on \cstar-algebras, a number of researchers (predominantly from the Japanese community) have introduced many sophisticated methods over the years to classify specific kinds of group actions utilizing certain Rokhlin-type properties \cite{Kishimoto95, EvansKishimoto97, Kishimoto98, Kishimoto98II, Nakamura00, Izumi04, Izumi04II, KatsuraMatui08, Matui08, Matui10, Matui11, IzumiMatui10, Sato10, Szabo21R, IzumiMatui21, IzumiMatui21_2}.
In direct comparison to the generality achieved for actions on von Neumann algebras, the implementation of the involved methods (in particular the Evans--Kishimoto intertwining argument) for actions on \cstar-algebras remained challenging beyond some specific classes of actions or acting groups.

To combat these methodological obstacles, the first named author introduced a categorical framework in \cite{Szabo21cc} to open up the classification of \cstar-dynamics up to cocycle conjugacy to methodology directly inspired by \cite{Elliott10}.
For actions on classifiable \cstar-algebras without traces, the so-called Kirchberg algebras \cite{Kirchberg95}, this idea led to the recent breakthrough in \cite{GabeSzabo22su, GabeSzabo22kp}.
The main result of said work implies that, given any countable amenable group $G$, any outer $G$-action on a Kirchberg algebra is uniquely determined by its $KK^G$-class up to cocycle conjugacy.

Although one might be tempted to guess that similar breakthrough results ought to be in reach for actions on finite classifiable \cstar-algebras, one still has a long way to go before such a goal can be achieved.
In analogy to the original obstacles to classify all simple nuclear \cstar-algebras \cite{Villadsen99, Rordam03, Toms08}, there are basic structural questions to be settled before a classification theory such as in \cite{GabeSzabo22kp} can be attempted on finite \cstar-algebras.
When concerned with just the underlying \cstar-algebras, this is already a serious challenge.
On the one hand, there is the question whether the \cstar-algebras under consideration automatically satisfy certain properties predicted by classification.
For the purpose of this article we highlight the property of Jiang--Su stability.
If $\mathcal Z$ is the so-called Jiang--Su algebra from \cite{JiangSu99}, then a \cstar-algebra $A$ is called Jiang--Su stable or $\mathcal Z$-stable when $A\cong A\otimes\mathcal Z$.
Although this might seem like a technical property at first glance, it becomes natural with more context: Firstly, $\mathcal Z$ behaves (as a \cstar-algebra) very much like an infinite-dimensional version of the complex numbers $\mathbb C$, for instance at the level of $K$-theory and traces.
Secondly, there is by now a pile of evidence that $\mathcal Z$-stability holds automatically for many \cstar-algebras arising from various constructions like the crossed product \cite{TomsWinter13, Kerr20, KerrSzabo20, KerrNaryshkin21}.
The discovery that $\mathcal Z$-stability does in fact \emph{not} hold automatically for all simple nuclear \cstar-algebras has, among other things, led to the nearly proven Toms--Winter conjecture, which asserts that $\mathcal Z$-stability can only hold or fail in conjunction with some other, a priori different, regularity conditions.

On the other hand, there is the question about precisely what additional structural consequences (not necessarily equivalent characterizations) are shared by Jiang--Su stable \cstar-algebras, a good example of which is the recent breakthrough work \cite{CETWW21} (which was in turn continuing work from \cite{MatuiSato14b, BBSTWW, SatoWhiteWinter15}).
The most novel technical achievement therein can be identified as the \emph{tracial local-to-global principle} for \cstar-algebras satisfying the so-called uniform property Gamma, which is a weaker assumption than Jiang--Su stability.
Said principle concerns the behavior of elements in a given \cstar-algebra $A$ with respect to the 2-seminorm $\|\cdot\|_{2,\tau}$ induced by individual tracial states $\tau$ on $A$ on the one hand, and the behavior with respect to the uniform tracial 2-norm $\|\cdot\|_{u}=\sup_{\tau} \|\cdot\|_{2,\tau}$ on the other hand.
While the latter is often of interest in the deeper structure and classification theory for \cstar-algebras, the former can be understood by studying the tracial von Neumann algebra $\pi_\tau(A)''$ arising as the weak closure of $A$ under the GNS representation associated to an individual trace $\tau$.
In a nutshell, the tracial local-to-global principle asserts that any suitable behavior that can be observed one trace at a time can also be observed uniformly, which often allows one to \emph{transfer}, so to speak, phenomena from von Neumann algebras to the \cstar-algebraic context.
This became in turn the main technical driving force behind the main results in \cite{CETWW21, CETW22}, which can be summarized by saying that the Toms--Winter conjecture holds for all simple nuclear \cstar-algebras having the uniform property Gamma.
Whether the latter property automatically holds for simple nuclear non-elementary \cstar-algebras is presently unknown, but is of high interest as it currently represents the main obstacle towards a full solution to the Toms--Winter conjecture.

When we turn our attention to \cstar-dynamics instead of \cstar-algebras, we can (and should) study analogous well-behavedness properties as for \cstar-algebras, one important example of which is equivariant Jiang--Su stability.
An action $\alpha: G\curvearrowright A$ on a \cstar-algebra is called (equivariantly) Jiang--Su stable or $\mathcal Z$-stable, if $\alpha$ is cocycle conjugate to $\alpha\otimes\operatorname{id}_{\mathcal Z}: G\curvearrowright A\otimes\mathcal Z$.
Assuming $G$ is amenable, it is presently open if this happens automatically when $A$ is simple nuclear and $\mathcal Z$-stable; cf.\ \cite[Conjecture A]{Szabo21si}.
We note that the analogous question for non-amenable groups is known to have a negative answer \cite{Jones83co, GardellaLupini21}, although recent insights as in \cite{Suzuki21, EGKNV22} leave some hope for the class of amenable actions of non-amenable groups, which we shall not investigate in this article.
If one stresses the point again that $\mathcal Z$ essentially looks like an infinite-dimensional version of $\mathbb C$, it should not come as a surprise that we can only expect a classification of $G$-actions on classifiable \cstar-algebras by reasonable invariants if they are equivariantly $\mathcal Z$-stable.
The existing work in this direction seems to indicate that equivariant Jiang--Su stability may indeed hold automatically whenever one can reasonably expect it \cite{MatuiSato12, MatuiSato14, Szabo18kp, Sato19, GardellaHirshbergVaccaro, Wouters21}.
The other possible line of investigation, namely further structural consequences of equivariant Jiang--Su stability for group actions, was initiated in \cite{GardellaHirshbergVaccaro} as a direct adaption of techniques in \cite{CETWW21}, albeit under rather restrictive assumptions on the actions.
We recall one of the key concepts from both said paper and the present work, but restrict ourselves in this introduction to the case of unital simple \cstar-algebras for convenience, despite actually investigating the concept in broader generality.\footnote{The reader might consult Definitions \ref{definition:tracial-norms}, \ref{definition:tracial_ultrapowers}, and \ref{definition:equivariantGamma} and compare with \cite[Definition 3.1]{GardellaHirshbergVaccaro}.}

\begin{definitioni}
Let $G$ be a countable discrete group.
Let $A$ be a separable unital simple \cstar-algebra such that all 2-quasitraces on $A$ are traces, and $T(A)\neq\emptyset$.
Given a free ultrafilter $\omega$ on $\mathbb N$, form the uniform tracial ultrapower $A^\omega$ of $A$.
An action $\alpha: G\curvearrowright A$ is said to have equivariant uniform property Gamma, if for any $k\geq 2$, there exist pairwise orthogonal projections $p_1,\dots,p_k\in (A^\omega\cap A')^{\alpha^\omega}$ such that
\[
\tau(ap_j)=\frac1k\tau(a) \quad\text{for all } j=1,\dots,k,\ a\in A,\ \tau\in T_\omega(A).
\]
\end{definitioni}

One can observe easily (Remark \ref{remark:Z-stability_implies_Gamma}) that equivariant uniform property Gamma is always implied by equivariant Jiang--Su stability.
This is important to note because it means that all non-trivial consequences of this property will also hold for Jiang--Su stable actions, and may in fact turn out to be useful for subsequent applications.
The most important technical consequence relevant to the present work is the tracial local-to-global principle for \cstar-dynamical systems over amenable groups.
An ad-hoc version of this principle appeared in \cite[Lemma 4.5]{GardellaHirshbergVaccaro}, but is only applicable (cf.\ \cite[Remark 2.2]{GardellaHirshbergVaccaro}) for actions that induce an action on tracial states with finite orbits of uniformly bounded size.
This assumption appeared not only as a prerequisite for the theory in \cite{GardellaHirshbergVaccaro}, but is implicitly crucial for the usefulness of the conclusion of this ad-hoc principle, which only involves tracial states that are fixed by the action.\footnote{This uses that if $\alpha: G\curvearrowright A$ is assumed to induce an action $G\curvearrowright T(A)$ with finite orbits of uniformly bounded cardinality $M>0$, then one has $\|\cdot\|_{2,u}\leq M\|\cdot\|_{2,T(A)^\alpha}$ as norms on $A$.}
In this article we aim to remove any assumptions about how actions $G\curvearrowright A$ are allowed to act on the traces of $A$, as well as strengthen the conclusion of the tracial local-to-global principle compared to \cite{GardellaHirshbergVaccaro}, in such a way as to directly generalize and strengthen the known principle for \cstar-algebras \cite[Lemma 4.1]{CETWW21}.
In addition to formulating our result in the language of $*$-polynomials as all prior papers did, we would also like to promote the following (formally equivalent) formulation of the tracial local-to-global principle for \cstar-dynamics, which becomes our main technical result.
We not only restrict ourselves for the moment to actions on unital simple nuclear \cstar-algebras (similarly to before) for convenience, but give a slightly weaker version here that is easier to state.
We treat a stronger version of the statement in broader generality in the main body of the paper; see Theorems \ref{theorem:local-to-global} and \ref{theorem:reformulation_local_to_global}.

\begin{theoremi} \label{theorem-B}
Let $G$ be a countable amenable group and $A$ a separable unital simple nuclear \cstar-algebra.
Let $\alpha: G\curvearrowright A$ be an action with equivariant uniform property Gamma.
Let $\delta: G\curvearrowright D$ be an action on a separable \cstar-algebra and $B\subseteq D$ a $\delta$-invariant \cstar-subalgebra.
Suppose that $\varphi: (B,\delta)\to (A^\omega,\alpha^\omega)$ is an equivariant $*$-homomorphism.
Then $\varphi$ extends to an equivariant $*$-homomorphism $\bar{\varphi}: (D,\delta)\to (A^\omega,\alpha^\omega)$ if and only if for every trace $\tau\in\overline{T_\omega(A)}^{w^*}$, there exists a $*$-homomorphism $\varphi^\tau: (D,\delta)\to (\pi_\tau^{\alpha^\omega}(A^\omega)'',\alpha^\omega)$ with $\varphi^\tau|_B=\pi_\tau^{\alpha^\omega}\circ\varphi$.\footnote{In actuality one may even allow $\varphi^\tau$ to have range in the tracial ultrapower of this von Neumann algebra, but this requires more cumbersome notation to state rigorously.}
Here $\pi_\tau$ denotes the GNS representation associated to the trace $\tau$ and $\pi_\tau^{\alpha^\omega}$ denotes the direct sum representation $\bigoplus_{g\in G} \pi_\tau\circ\alpha^\omega_{g^{-1}}$.
\end{theoremi}

Not unlike previous approaches, the proof of this main result factors through a kind of dynamical version of CPoU (the existence of so-called \emph{complemented partitions of unity}) that we establish along the way; see Lemma \ref{lemma:equivariant_CPOU}.
There are two things to note about this, however.
Firstly, the present version of dynamical CPoU does not generally match the property suggested for this purpose in \cite{GardellaHirshbergVaccaro}, and based on our work we are in fact uncertain whether that property can be expected to hold even under the validity of the above theorem.
Secondly, we would like to propose a slight perspective shift by viewing the above local-to-global principle as the primary conceptual property to be studied and exploited instead of the dynamical CPoU, which we feel --- especially compared to previous iterations --- to be rather unwieldy by itself due to its elaborate technical nature.

In the main body of the paper, we actually prove a stronger version of Theorem \ref{theorem-B} for a class of actions on much more general \cstar-algebras.
We would like to comment that the starting point of our theory presumes the underlying \cstar-algebra to satisfy a kind of \emph{weak CPoU}, namely the one shown to hold for nuclear \cstar-algebras in \cite[Lemma 3.6]{CETWW21}.
Fortunately, a result in the recent preprint \cite{CCEGSTW} implies that this kind of weak CPoU in fact holds automatically for all \cstar-algebras with compact tracial state space, which we can use to our advantage.

As for the rest of the paper, we apply Theorem \ref{theorem-B} (or rather Theorems \ref{theorem:local-to-global} and \ref{theorem:reformulation_local_to_global}) to gain insight on equivariant Jiang--Su stability.
A famous argument due to Matui--Sato \cite{MatuiSato12} and the main result of \cite{Szabo21si} allows us to argue (as explained in Section 5) that an action $\alpha$ as above is equivariantly Jiang--Su stable if and only if $A\cong A\otimes\mathcal Z$ and $\alpha$ is uniformly McDuff, i.e., there exist unital $*$-homomorphisms $M_n\to (A^\omega\cap A')^{\alpha^\omega}$ for all $n \in \N$.
Once we note that the latter property is known to hold one trace at a time as a consequence of Ocneanu's theorem \cite{Ocneanu85} (in the generality we need it, this is imported from \cite{SzaboWouters23md}), the above result can be applied to the $\alpha$-equivariant inclusion $1_n\otimes\operatorname{id}_A: A\to M_n(A)$ to deduce the following consequence.
As before, we note that we prove this result in greater generality than stated here; see Theorem \ref{theorem:Gamma-equivalent-to-Z-stable}.

\begin{corollaryi}[cf.\ {\cite[Theorem 7.6]{GardellaHirshbergVaccaro}}]
Let $\alpha: G\curvearrowright A$ be an action of a countable amenable group on a separable unital simple nuclear $\mathcal Z$-stable \cstar-algebra.
Then $\alpha$ has equivariant uniform property Gamma if and only if $\alpha$ is equivariantly Jiang--Su stable.
\end{corollaryi}

We expect the main result of this article to have impact on subsequent applications of equivariant uniform property Gamma or equivariant Jiang--Su stability, in particular in the context of classifying actions on tracial \cstar-algebras.

As far as potential further research is concerned, let us point out that for group actions $\alpha: G\curvearrowright A$ that are assumed to be `sufficiently free'\footnote{A priori, this may have several different interpretations.}, the theory pursued in this article can be seen as an instance where one studies uniform property Gamma for the inclusion of \cstar-algebras $A\subseteq A\rtimes_{\alpha,r} G$, in such a way as to strengthen uniform property Gamma for $A$.
It is a tantalizing issue to determine a common framework encompassing all applications of interest regarding uniform property Gamma for more general inclusions of \cstar-algebras.
For instance, it has been hypothesized in past work \cite[Remark 9.6]{KerrSzabo20} that for a free minimal action $G\curvearrowright X$ of an amenable group on a compact metric space, some desirable dynamical properties ought to follow from a different kind of uniform property Gamma for the inclusion $\mathcal C(X)\subseteq\mathcal{C}(X)\rtimes G$, namely the one that strengthens uniform property Gamma for the crossed product; see also \cite{LiaoTikuisis22}.

Lastly, we would like to express our sincere gratitude to the anonymous reviewers for their valuable comments that helped improve this article.

\section{Preliminaries}

\begin{notation} \label{nota:basic-notation}
Throughout this paper, we will use the following notations and conventions unless specified otherwise:
\begin{itemize}[itemsep=2pt]
\item By default, $\omega$ denotes some free ultrafilter on $\N$.
At times it can make it easier to state a claim using two free ultrafilters, in which case we denote a second one by $\kappa$.
\item If $F$ is a finite subset inside another set $M$, we often denote this by $F \ssubset M$.
\item $\mathbb{K}$ denotes the compact operators on the  Hilbert space $\ell^2(\N)$. 
\item Let $A$ be a \cstar-algebra. We denote its positive elements by $A_+$ and its minimal unitization by $\tilde{A}$. We will also make use of its Pedersen ideal, denoted by $\mathcal{P}(A)$.
We assume the reader is familiar with the basic properties of this object.
Given a positive element $a \in A$ and $\e>0$, we denote by $(a-\e)_+$ the positive part of the self-adjoint element $a-\e1_{\tilde{A}}$.
\item The topological cone of lower semicontinuous traces on $A_+$ will be denoted by $\tilde{T}(A)$; cf.\ \cite{ElliottRobertSantiago11}.
We call such a trace $\tau$ on $A$ \emph{trivial} if it is $\{0,\infty\}$-valued.
It is well-known that trivial traces are in one-to-one correspondence with the ideal lattice of $A$ by mapping a trivial trace $\tau$ to the linear span of $\tau^{-1}(0)$.
The set of non-trivial lower semicontinuous traces on $A_+$ will be denoted by $T^+(A)$ and the set of tracial states will be denoted by $T(A)$. 
In this paper, we say that a compact subset $K \subset T^+(A)$ is a \emph{compact generator} for $T^+(A)$ if $\R^{>0} K = T^+(A)$.\footnote{In case $A$ is simple, an example of such a compact generator is given by $\{\tau \in T^+(A) \mid \tau(a) = 1\}$ for some $a \in \mathcal{P}(A)_+\setminus\{0\}$.}
\item In addition, we denote by $Q\tilde{T}_2(A)$ the set of lower semicontinuous 2-quasitraces (see \cite[Definition~2.22]{BlanchardKirchberg04}) on $A$, which contains $\tilde{T}(A)$.
We usually only mention them to assume in appropriate contexts that there are no genuine quasitraces, i.e., $Q\tilde{T}_2(A)=\tilde{T}(A)$.
\end{itemize}
\end{notation}

We recall the following existence theorem for traces. This follows from a combination of the work of Blackadar--Cuntz \cite[Theorem 1.5]{BlackadarCuntz82} and Haagerup \cite{Haagerup14} (see also \cite[Remark~2.29(i)]{BlanchardKirchberg04}).
\begin{theorem}
Let $A$ be a simple, exact \cstar-algebra such that $A \otimes \mathbb{K}$ contains no infinite projections.
Then $Q\tilde{T}_2(A)=\tilde{T}(A)$ and $T^+(A) \neq \emptyset$. 
\end{theorem}

In particular, this implies that each stably finite, simple, separable, nuclear $C^*$-algebra admits a non-trivial trace.

\begin{definition}[{\cite[Definition 1.1]{Kirchberg04}, \cite[Definition 4.3]{KirchbergRordam14}}]\label{definition:ultrapowers}
Let $A$ be a \cstar-algebra with an action $\alpha:G \acts A$ of a discrete group.
\begin{enumerate}
\item The \textit{ultrapower} of $A$ is defined as
\[A_\omega := \ell^\infty(A)/\{(a_n)_{n \in \N} \in \ell^\infty(A) : \lim_{n \rightarrow \omega} \|a_n\| = 0\}.\]
\item Pointwise application of $\alpha$ on representing sequences induces an action on the ultrapower, which we will denote by $\alpha_\omega: G \acts A_\omega$.
\item There is a natural inclusion $A \subset A_\omega$ by identifying an element of $A$ with its constant sequence. Define
\[A_\omega \cap A' := \{x \in A_\omega\mid [x,A]=0\} \quad \text{and}\quad A_\omega \cap A^\perp := \{x \in A_\omega \mid xA=Ax=0\}.\]
The quotient
\[F_\omega(A) := (A_\omega \cap A')/(A_\omega \cap A^\perp)\]
is called the \emph{(corrected) central sequence algebra}. If $A$ is $\sigma$-unital, then $F_\omega(A)$ is unital, where the unit is represented by a sequential approximate unit $(e_n)_{n\in \N}$.
\item Since $A$ is $\alpha_\omega$-invariant, so are $A_\omega \cap A'$ and $A_\omega \cap A^\perp$. Thus, $\alpha_\omega$ induces an action on $F_\omega(A)$, which we will denote by $\tilde{\alpha}_\omega: G \acts F_\omega(A)$.
\end{enumerate}
\end{definition}

\begin{definition}
Let $A$ be a \cstar-algebra. 
A sequence of tracial states $(\tau_n)_{n\in \N}$ on $A$ defines a trace on $A_\omega$ via
\[
[(a_n)_{n \in \N}] \mapsto \lim_{n \rightarrow \omega} \tau_n(a_n).\]
A trace of this form is called a \emph{limit trace}. The set of all limit traces on $A_\omega$ will be denoted by $T_\omega(A)$. 
More generally, following {\cite[Definition 2.1]{Szabo21si}}, a sequence $(\tau_n)_{n \in \N}$ in $\tilde{T}(A)$ defines a lower semicontinuous trace $\tau: \ell^\infty(A)_+ \rightarrow [0,\infty]$ by
\[
\tau((a_n)_{n \in \N}) = \sup_{\e >0} \lim_{n \rightarrow \omega} \tau_n((a_n- \e)_+).
\]
This trace is the lower semicontinuous regularization of the trace given by $\lim_{n \rightarrow \omega} \tau_n(a_n)$, see \cite[Lemma 3.1]{ElliottRobertSantiago11}.
This regularization ensures that $\tau((a_n)_{n\in\N}) =0$ if ${\lim_{n \rightarrow \omega} \|a_n\| = 0}$, so $\tau$ also induces a lower semicontinuous trace on $A_\omega$.
A trace of this form on $A_\omega$ is called a \emph{generalized limit trace}.
The set of all generalized limit traces is denoted by $\tilde{T}_\omega(A)$.

For the next part, assume $A$ is separable. For any $a \in A_+$ and $\tau \in \tilde{T}_\omega(A)$ we can define a trace
\[
\tau_a: (A_\omega \cap A')_+ \rightarrow [0,\infty],\quad x \mapsto \tau(ax).
\]
We have that $\tau_a(x) \leq \|x\|\tau(a)$, so this trace is bounded whenever $\tau(a) < \infty$.
Note that this trace also induces a trace on $F_\omega(A)$, which by abuse of notation will also be denoted by $\tau_a$. 
Clearly this yields a tracial state under the assumption $\tau(a)=1$.
Let us say that a given tracial state $\tau$ on $F_\omega(A)$ is a \emph{canonical trace}, if it belongs to the weak-$*$-closed convex hull of $\{ \tau_a \mid \tau\in\tilde{T}_\omega(A),\ a\in A_+,\ \tau(a)=1 \}$.
\end{definition}

\begin{remark}\label{remark:generalized_limit_traces_bounded_case}
We point out that it is not necessary to consider generalized limit traces in an important subcase that often occurs in the literature.
Namely, assume $A$ is a separable simple \cstar-algebra with $\emptyset\neq T^+(A)=\R^{>0} T(A)$ such that $T(A)$ is compact.\footnote{For instance this is automatic when $A$ is separable simple nuclear unital and stably finite.}
Then it follows by \cite[Proposition 2.3]{CastillejosEvington21} that every generalized limit trace $\tau \in \tilde{T}_\omega(A)$ that is finite on some non-zero positive element of $A$ is a multiple of an ordinary limit trace.  
\end{remark}

Next we recall how various versions of tracial ultrapowers are defined.

\begin{definition}[see {\cite[Propositions 3.1, 3.2]{AndoHaagerup14}}]
Suppose $\mathcal{M}$ is a finite von Neumann algebra with faithful normal tracial state $\tau$.
Then the \emph{tracial von Neumann algebra ultrapower} is defined as
\begin{equation} \label{definition:vonNeumann_ultrapower}
\mathcal{M}^\omega := \ell^\infty(\mathcal{M})/\{(x_n)_{n \in \N} \in \ell^\infty(\mathcal{M}) \mid \lim_{n \rightarrow \infty} \|x_n\|_{2,\tau} = 0\}.
\end{equation}
This is again a von Neumann algebra with a faithful normal tracial state $\tau^\omega$ that is defined on representative sequences by $\tau^\omega((x_n)_{n\in\N}) = \lim_{n \rightarrow \omega} \tau(x_n)$.
\end{definition}

The notation used for the tracial von Neumann algebra ultrapower is the same as for the uniform tracial ultrapower of a suitable $C^*$-algebra as defined below.
It will be clear from context which of the two notions we use.
In the special case that $A$ is a $C^*$-algebra with unique tracial state $\tau$ and no unbounded traces, the uniform tracial ultrapower $A^\omega$ is naturally isomorphic, by Kaplansky's density theorem, to the von Neumann tracial ultrapower $(\pi_\tau(A)'')^\omega$, where $\pi_\tau$ denotes the GNS representation associated to $\tau$.
Note that on a tracial von Neumann algebra $(\mathcal{M},\tau)$, the topology induced by the $\|\cdot\|_{2,\tau}$-norm agrees with the $*$-strong operator topology on bounded subsets.
So equivalently, in \eqref{definition:vonNeumann_ultrapower} one can quotient out by the sequences that converge to 0 in the $*$-strong operator topology, which makes the construction equivalent to the Ocneanu ultrapower; cf.\ \cite{AndoHaagerup14}.

\begin{definition} \label{definition:tracial-norms}
Let $A$ be a \cstar-algebra.
Given a constant $p\geq 1$ and $\tau \in T(A)$, we define a seminorms $\|\cdot\|_{p,\tau}$ on $A$ by
\[
\|a\|_{p,\tau} = \tau(|a|^p)^{1/p},\quad a\in A.
\] 
We will in particular appeal to the cases $p=1$ or $p=2$ subsequently.
For a non-empty set $X \subset T(A)$, we define a seminorm $\|\cdot\|_{2,X}$ on $A$ by
\[\|a\|_{2,X} := \sup_{\tau \in X}\|a\|_{2,\tau}\]
for all $a \in A$.
The seminorm $\|\cdot\|_{2,T(A)}$ is also denoted by $\|\cdot\|_{2,u}$. This is a norm if and only if for all non-zero $a \in A$ there exists some $\tau \in T(A)$ such that $\tau(a^*a) >0$, which is in particular the case when $A$ is simple with $T(A)$ non-empty. 
\end{definition}

We note that in the construction below, we deviate from other sources by making a very explicit distinction in terminology between \cstar-algebras that do or do not admit non-trivial unbounded traces.

\begin{definition}[cf.\ {\cite[Paragraph 1.3]{CETWW21}}\footnote{The cited source assumes separability, but we generalize the definition beyond that case.}] \label{definition:tracial_ultrapowers}
Let $A$ be a \cstar-algebra with $T(A) \neq \emptyset$.
Then the \emph{trace-kernel ideal (with respect to bounded traces)} inside $A_\omega$ is defined by
\[
J_A^{\rm{b}} := \{ [(a_n)_{n \in \N}] \in A_\omega \mid \lim_{n \rightarrow \omega} \|a_n\|_{2,T(A)}= 0\}.
\]
The \emph{uniform bounded tracial ultrapower} is defined as the quotient
\[
A^{\omega,\rm b} := A_\omega/ J_A^{\rm{b}}.
\]
Whenever $\|\cdot\|_{2,T(A)}$ defines a norm on $A$, there also exists a canonical embedding of $A$ into $A^{\omega,\rm b}$.
Then $A^{\omega,\rm b} \cap A'$ is called the \emph{uniform bounded tracial central sequence algebra}. Whenever we have an action $\alpha: G \acts A$ of a discrete group, the ideal $J_A^{\rm{b}}$ is $\alpha_\omega$-invariant.
Hence, there is an induced action on the uniform bounded tracial ultrapower, which we will denote by $\alpha^\omega: G \acts A^{\omega,\rm b}$. 

Clearly, every limit trace vanishes on $J_A^{\rm b}$ and hence also induces a tracial state on $A^{\omega,\rm b}$.
We will also use $T_\omega(A)$ to denote the collection of limit traces on $A^{\omega,\rm b}$.
Note that 
\[
J_A^{\rm{b}} =\{ x \in A_\omega : \|x\|_{2,T_\omega(A)} = 0\},
\]
so in particular $\|\cdot\|_{2,T_\omega(A)}$ defines a norm on $A^{\omega,\rm b}$.

Finally, if we assume $A$ is a simple \cstar-algebra such that $Q\tilde{T}_2(A)=\tilde{T}(A)$ and $\emptyset \neq T^+(A)=\mathbb{R}^{>0}\cdot T(A)$ with $T(A)$ compact, then we simply call $J_A=J_A^{\rm{b}}$ the trace-kernel ideal, $A^\omega=A^{\omega,\rm b}$ the uniform tracial ultrapower, and $A^\omega\cap A'$ the uniform tracial central sequence algebra.
\end{definition}

\begin{remark} \label{rem:weird-examples-for-Gamma}
Our choice to add the extra ``bounded'' in the terminology above and the extra letter ``$\rm b$'' in the notation, which is usually not included in other sources such as the ones we cite, is deliberate and has the purpose to not overuse the word ``uniform'', in particular in cases where it becomes rather misleading.
This is most apparent for non-simple \cstar-algebras; if $B$ is any unital simple \cstar-algebra with $T(B)\neq\emptyset$, then the above construction applied to $A=B\oplus\mathbb{K}$ yields $A^{\omega,\rm b}=B^\omega$ by virtue of the fact that the canonical trace on $\mathbb{K}$ is unbounded.
Since one of the two tracial direct summands is entirely forgotten in this construction, this object seems unfit to be called ``uniform tracial''.
However, even the case of simple \cstar-algebras is enough to illustrate why one should not equate $A^{\omega,\rm b}$ with the object capturing all ``uniform tracial'' data.
Namely, the range result \cite{GongLin22} combined with a little playing around with invariants allows one to see that given any metrizable Choquet simplex $S$ with $\partial_e S$ admitting some isolated point, there exists a (non-unital) classifiable \cstar-algebra $A$ such that $T^+(A)$ has a Choquet base affinely homeomorphic to $S$, yet $A$ has a unique tracial state $\tau$.
In this scenario, we have $A^{\omega,\rm b}\cong (\pi_\tau(A)'')^\omega\cong \mathcal{R}^\omega$ as a consequence of Connes' theorem.
So despite $A$ having a rich tracial structure, the only trace captured by this construction is $\tau$, which compels us to not apply the word ``uniform'' or the notation ``$A^\omega$'' to such an example.

Note that the phenomenon discussed here is also what motivated us to subsequently revise the definition of (equivariant) uniform property Gamma in the spirit of \cite{CastillejosEvington21}, as well as introduce an auxiliary version of it that explicitly only takes into account tracial states, even when the surrounding \cstar-algebra may have other unbounded traces.
\end{remark}

\begin{remark}\label{remark:tracial_ultrapower_unital_T(A)_compact}
Let $A$ be a $\sigma$-unital \cstar-algebra with $T(A)\neq\emptyset$.
By \cite[Proposition 1.11]{CETWW21} the uniform bounded tracial ultrapower $A^{\omega,\rm b}$ is unital if and only if $T(A)$ is compact.\footnote{The cited statement assumes separability of $A$, but a closer look at the proof shows that $\sigma$-unitality is sufficient.}
Moreover, \cite[Lemma 1.10]{CETWW21} shows that in that case the natural map ${A_\omega \cap A' \rightarrow A^{\omega,\rm b} \cap A'}$ factors through $F_\omega(A)$. In case $A$ is separable, this natural map is surjective by a combination of \cite[Propositions 4.5(iii) and 4.6]{KirchbergRordam14} (the unitality hypothesis in the second cited proposition is not needed, as it suffices to take a unit in the minimal unitisation for the proof).
\end{remark}

As we have argued above, there are some issues if one is trying to define the object $A^\omega$ for a \cstar-algebra $A$ that may possess many unbounded traces.
In fact, trying to find a viable general definition that has the same level of utility as in the case of unital \cstar-algebras has eluded a number of researchers for years.
By introducing the next few definitions and observations, however, we wish to promote the viewpoint that there is a rather natural way to define the object $A^\omega\cap A'$ for any separable \cstar-algebra $A$, even if we do not know at present how to properly define the object $A^\omega$ itself.
Since we are unsure of the viability of this definition when $A$ admits genuine quasitraces, we wish to be cautious and shall only define the concepts below under the assumption that $A$ does not admit them.\footnote{At the same time, we note that the concepts make sense formally anyway, and none of the subsequent arguments hinge on the assumption that $A$ does not admit genuine quasitraces.}

\begin{definition} \label{def:AomegacapAprime}
Let $A$ be a separable \cstar-algebra with $Q\tilde{T}_2(A)=\tilde{T}(A)$ and $T^+(A)\neq\emptyset$.
The \emph{trace-kernel ideal} $\mathcal{J}_A$ inside $F_\omega(A)$ is defined as the set of elements $x\in F_\omega(A)$ such that for every generalized limit trace $\tau\in\tilde{T}_\omega(A)$ and $a\in A_+$ with $0<\tau(a)<\infty$, we have $\tau_a(x^*x)=0$.
With some abuse of notation, we denote the quotient by \begin{equation}\label{eq:AomegacapAprime} A^\omega\cap A' = F_\omega(A)/\mathcal{J}_A.\end{equation}
It is clear from construction that a canonical trace on $F_\omega(A)$ vanishes on $\mathcal{J}_A$, so it descends to a tracial state on $A^\omega\cap A'$.
As before, we call a given tracial state on $A^\omega\cap A'$ a \emph{canonical trace}, if it is induced by a canonical trace on $F_\omega(A)$, or equivalently, if it belongs to the weak-$*$-closed convex hull of the tracial states $\tau_a$ on $A^\omega\cap A'$, where $\tau\in \tilde{T}_\omega(A)$ and $a\in A_+$ with $\tau(a)=1$.
 
If $\alpha: G\acts A$ is an action of a discrete group with induced action $\tilde{\alpha}_\omega: G\acts F_\omega(A)$, then clearly $\mathcal{J}_A$ is $\tilde{\alpha}_\omega$-invariant, so we obtain an induced action $\alpha^\omega: G\acts A^\omega\cap A'$.
\end{definition}

\begin{remark}\label{remark:consistent_notation_ultrapower}
In case $A$ is simple and that $Q\tilde{T}_2(A)=\tilde{T}(A)$ and $\emptyset \neq T^+(A)=\mathbb{R}^{>0} T(A)$ with $T(A)$ compact, it follows from Remarks \ref{remark:generalized_limit_traces_bounded_case} and \ref{remark:tracial_ultrapower_unital_T(A)_compact} that $F_\omega(A)/\mathcal{J}_A = A^{\omega, \mathrm{b}}\cap A'$, so the notation $A^\omega \cap A'$ is consistent with the last part of Definition \ref{definition:tracial_ultrapowers}.
\end{remark}

\begin{remark}[see remark after {\cite[Definition 4.3]{KirchbergRordam14}}]
Let $p\geq 1$ be any constant.
Given any element $x$ in a \cstar-algebra $B$ with a tracial state $\theta$, one has the inequalities 
\[
\|x\|_{1,\theta}\leq\|x\|_{p,\theta}\leq\|x\|_{1,\theta}^{1/p}\|x\|^{1-1/p}.
\]
This implies that an element in either Definition \ref{definition:tracial_ultrapowers} or \ref{def:AomegacapAprime} belongs to the trace-kernel ideal if and only if its tracial $p$-norms vanish with respect to the appropriately chosen (limit) traces.
We will frequently use this without further mention for $p=1$.
\end{remark}

\begin{remark}
One of Kirchberg's initial observations about $F_\omega(A)$, which attests to the naturality of its construction, is that it is a stable invariant.
We are about to argue that the same applies to the construction $A\mapsto A^\omega\cap A'$.
For this purpose, let $\{ e_{k,\ell}\mid k,\ell\geq 1\}$ be a set of matrix units generating $\mathbb K$, and let $1_n\in\mathbb K$ be the increasing approximate unit given by $1_n=\sum_{j=1}^n e_{j,j}$.
We recall (cf.\ \cite[Proposition 1.9, Corollary 1.10]{Kirchberg04}) that there is a canonical isomorphism $\theta: F_\omega(A)\to F_\omega(A\otimes\mathbb K)$ defined as follows: given an element $x\in F_\omega(A)$ represented by a central sequence $(x_n)_{n \in \N}$ in $A$, it is sent to the element $\theta(x)$ represented by the central sequence $(x_n\otimes 1_n)_{n \in \N}$.
\end{remark}

\begin{prop} \label{prop:tracial-central-sequence-algebra-stable}
Let $A$ be a separable \cstar-algebra with $Q\tilde{T}_2(A)=\tilde{T}(A)$ and $T^+(A)\neq\emptyset$.
Then the canonical isomorphism $F_\omega(A)\cong F_\omega(A\otimes\mathbb K)$ preserves the canonical traces on both sides.
Consequently, it descends to a canonical isomorphism
\[
A^\omega\cap A' \cong (A\otimes\mathbb K)^\omega\cap (A\otimes\mathbb K)'.
\]
\end{prop}
\begin{proof}
As we set up before the Proposition, we denote the canonical isomorphism by $\theta$.
It is clear that it induces an affine homeomorphism between all tracial states on $F_\omega(A)$ and on $F_\omega(A\otimes\mathbb K)$ via $\tau\mapsto\tau\circ\theta^{-1}$.
The claim amounts to showing that the image of the canonical traces on the left is equal to the canonical traces on the right.

Let $\rm Tr$ be the unique lower semicontinuous trace on $\mathbb K$ with $\mathrm{Tr}(e_{1,1})=1$.
We keep in mind that the assignment $\tilde{T}(A)\to\tilde{T}(A\otimes\mathbb K)$ given by $\tau\mapsto \tau\otimes\rm Tr$ is an affine homeomorphism.
Given a generalized limit trace $\tau\in\tilde{T}_\omega(A)$ induced by a sequence $(\tau_n)_{n \in \N}$ in $\tilde{T}(A)$, let us denote by $\tau^s\in\tilde{T}_\omega(A\otimes\mathbb K)$ the generalized limit trace induced by the sequence $(\tau_n\otimes\mathrm{Tr})_{n \in \N}$ in $\tilde{T}(A\otimes\mathbb K)$.
Clearly the assignment $\tau\mapsto\tau^s$ is also a bijection between generalized limit traces. Let such a generalized limit trace $\tau$ be given on $A_\omega$.
Given $a \in (A\otimes\mathbb K)_+$, we can write $a=\sum_{k,\ell=1}^\infty a_{k,\ell}\otimes e_{k,\ell}$ for uniquely determined elements $a_{k,\ell}\in A$.
It then follows from \cite[Proposition 2.9]{CastillejosEvington21} that we have a norm-convergent sum expression
\begin{equation}\label{eq:canonical_isomorphism_traces}
\tau^s_a\circ\theta = \sum_{\ell=1}^\infty \tau_{a_{\ell,\ell}}.
\end{equation}
Applied to $a = b \otimes e_{1,1}$ for some $b \in A_+$ with $\tau(b) =1$, this gives $\tau_{b} \circ \theta^{-1} = \tau_a^s$.
From this we can infer that canonical traces are mapped to canonical traces. The general expression \eqref{eq:canonical_isomorphism_traces} applied to $a \in (A\otimes\mathbb K)_+$ with $\tau^s(a)=1$ shows that we have a bijection.
\end{proof}

We give two more technical lemmas that will be useful later on.

\begin{lemma}\label{lemma:approximation_traces}
Let $A$ be a \cstar-algebra with positive element $a \in A_+$.
Let $(\varepsilon_n)_{n \in \N}$ be a sequence of positive constants such that $\lim_{n \rightarrow \infty} \varepsilon_n = 0$ and $(b_n)_{n\in\N}$ a sequence of positive elements such that $\|b_n - a\| < \varepsilon_n$.
For each $\tau \in \tilde{T}_\omega(A)$ and $c \in F_\omega(A)$ one has
\[
\lim_{n \rightarrow\infty} \tau_{(b_n-\varepsilon_n)_+}(c) = \tau_{a}(c). 
\] 
\end{lemma}
\begin{proof}
For each $n \in \N$ there exists a contraction $d_n \in A$ such that $(b_n - \varepsilon_n)_+ = d_nad_n^*$ by \cite[Lemma 2.2]{KirchbergRordam02}.
This implies that  $\tau_{(b_n-\varepsilon_n)_+}(c) \leq \tau_{a}(c)$.
Since the sequence $(b_n - \varepsilon_n)_+$ converges to $a$ and $\tau$ is lower semicontinuous, this leads to the desired result. 
\end{proof}

\begin{lemma}\label{lemma:compact_generator_no_cutdowns}
Let $A$ be a simple \cstar-algebra with $a \in \mathcal{P}(A)_+\setminus\{0\}$ and let $K$ be a compact generator for $T^+(A)\neq\emptyset$.
Take a generalized limit trace $\tau \in \tilde{T}_\omega(A)$ such that ${0 < \tau(a) < \infty}$.
Then there exists a sequence $(\theta_n)_{n \in \N}$ in $K$ such that the associated generalized limit trace $\theta$ on $A_\omega$ is a scalar multiple of $\tau$ and such that for each sequence $(b_n)_{n\in \N}$ representing an element of $A_\omega$ we have that
\begin{equation}\label{eq:no_cutdowns}\theta((ab_n)_{n \in \N}) =  \lim_{n \rightarrow \omega} \theta_n(ab_n).\end{equation}
\end{lemma}
\begin{proof}
The fact that there exists a sequence $(\theta_n)_{n \in \N}$ in $K$ such that the associated generalized limit trace is a multiple of $\tau$ follows directly from \cite[Lemma 2.10 and Remark 2.11]{Szabo21si}.
Let $B := \overline{a^{1/2}Aa^{1/2}}$ denote the hereditary subalgebra generated by $a^{1/2}$. Then $B \subseteq \mathcal{P}(A)$, see for example \cite[Proposition~5.6.2]{Pedersen79}.
Since $K$ is compact we claim that $\sup_{\sigma \in K} \|\sigma \big \lvert_B\| < \infty$. Suppose that this would not be the case, then for all $n \in \N$ we could find a $\sigma_n \in K$ and a positive contraction $d_n \in B$ such that $\sigma_n(d_n) \geq n 2^n$. Consider $d := \sum_{n=1}^\infty 2^{-n}d_n \in B$ and $\sigma = \lim_{n \rightarrow \omega} \sigma_n \in K$ (using compactness of $K$). Then for each $n \in \N$ we would get 
\[\sigma(d) = \lim_{n \rightarrow \omega} \sigma_n(d) \geq  \lim_{n \rightarrow \omega} \sigma_n(2^{-n} d_n) = \infty,\]
 but this is a contradiction since $d$ belongs to the Pedersen ideal. As a consequence we get that, when restricted to the hereditary subalgebra $\overline{a^{1/2}\ell^\infty(A)a^{1/2}}\subseteq\ell^\infty(B)$, the trace formed by $\lim_{n \rightarrow \omega} \theta_n$ is already bounded and hence continuous, so formula \eqref{eq:no_cutdowns} holds.
\end{proof}

The following proposition is a useful lifting property in various contexts.
The proof relies on the concept of $G$-$\sigma$-ideals, see \cite[Definition 4.1]{Szabo18ssa2}.
Let $\alpha: G\acts A$ and $\beta: G\acts B$ be actions on \cstar-algebras.
As in \cite{Kirchberg04}, we call an equivariant surjective $*$-homomorphism $\pi: (A,\alpha)\to (B,\beta)$ strongly locally semisplit, if for every separable $\beta$-invariant \cstar-subalgebra $D\subseteq B$, there exists an equivariant c.p.c.\ order zero map $\phi: (D,\beta)\to (A,\alpha)$ such that $\pi \circ \phi =\operatorname{id}_D$.

\begin{prop} \label{prop:fixed_point_surjectivity}
Let $A$ be a separable simple \cstar-algebra with $Q\tilde{T}_2(A)=\tilde{T}(A)$ and $T^+(A)\neq\emptyset$.
Let $\alpha: G \acts A$ be an action of a countable discrete group. 
Then the quotient map
\[
( F_\omega(A), \tilde{\alpha}_\omega) \rightarrow (A^{\omega} \cap A', \alpha^\omega)
\]
is strongly locally semisplit.
\end{prop}
\begin{proof}
By \cite[Proposition 4.5(ii)]{Szabo18ssa2}, it suffices to prove that $\mathcal{J}_A\subset F_\omega(A)$ is a $G$-$\sigma$-ideal.
Fix an element $0\neq a\in\mathcal{P}(A)_+$ and a compact generator $K\subset T^+(A)$.\footnote{As pointed out in the footnote after defining compact generators in \ref{nota:basic-notation}, this always exists as a consequences of simplicity.}
By \cite[Proposition 2.4]{Szabo21si} and Lemma \ref{lemma:approximation_traces}, we can conclude that $\mathcal{J}_A$ coincides with the ideal of those elements $x\in F_\omega(A)$ such that ${\tau_a(x^*x)=0}$ for all $\tau\in\tilde{T}_\omega(A)$ induced by any sequence $\tau_n\in K$.
Since $a$ belongs to the Pedersen ideal and $K$ is compact, this further implies that an element $x\in F_\omega(A)$ represented by a sequence $(x_n)_{n \in \N}$ in $A$ belongs to $\mathcal{J}_A$ precisely when $\lim_{n\to\omega} \max_{\tau\in K} \tau(a^{1/2}x_n^*x_na^{1/2}) = 0$.

We proceed to show that $\mathcal{J}_A$ is a $G$-$\sigma$-ideal.
Let $D\subset F_\omega(A)$ be a separable $\tilde{\alpha}_\omega$-invariant \cstar-subalgebra.
Let $(d_{k,n})_{n,k\in \N}, (c_{k,n})_{n,k\in \N}$ be two bounded double sequences in $A$ such that for each $k\in \N$, the sequences $(d_{k,n})_{n\in \N}$ and $(c_{k,n})_{n\in \N}$ are approximately central, the set $\{ d^{(k)}=[(d_{k,n})_{n\in\N}] \mid k\in \N \}$ defines a dense subset in the unit ball of $D$, and the set $\{ c^{(k)}=[(c_{k,n})_{n\in\N}] \mid k\in \N \}$ defines a dense subset in the unit ball of $D\cap\mathcal{J}_A$.
By Kasparov's Lemma \cite[Lemma 1.4]{Kasparov88}, we can find for any $\e>0$, $F\ssubset G$ and $m\in\N$ a positive element $e\in\mathcal{J}_A$ such that
\[
\max_{k\leq m} \|[e,d^{(k)}]\|\leq\e,\ \max_{k\leq m} \|(1-e)c^{(k)}\|\leq\e, \text{ and } \max_{g\in F} \|e-\tilde{\alpha}_\omega(e)\|\leq\e.
\]
Let $b\in A$ be a strictly positive contraction.
If we represent $e$ by an approximately central sequence $(e_n)_{n \in \N}$ of positive contractions in $A$, then it follows that
\[
\max_{k\leq m} \lim_{n\to\omega} \|[e_n,d_{k,n}]b\|\leq\e,\ \max_{k\leq m} \lim_{n\to\omega} \|(1-e_n) c_{k,n} b\|\leq\e,
\]
and
\[
\max_{g\in F} \lim_{n\to\omega} \|(e_n-\alpha_g(e_n))b\|\leq\e ,\ \lim_{n\to\omega} \max_{\tau\in K} \tau(a^{1/2}e_na^{1/2}) = 0.
\]
Appealing to Kirchberg's $\e$-test \cite[Lemma 3.1]{KirchbergRordam14}, we can find another approximately central sequence $(e_n)_{n \in \N}$ of positive contractions in $A$ satisfying the stronger property
\[
\lim_{n\to\omega} \big( \|[e_n,d_{k,n}]b\| + \|(1-e_n) c_{k,n} b\| + \|(e_n-\alpha_g(e_n))b\| \big) =0 \text{ and } \lim_{n\to\omega} \max_{\tau\in K} \tau(a^{1/2}e_na^{1/2}) = 0
\]
for all $k\in \N$ and $g\in G$.
This means that this sequence represents a positive contraction $e\in(\mathcal{J}_A\cap D')^{\tilde{\alpha}_\omega}$ such that $ec=c$ for all $c\in\mathcal{J}_A\cap D$.
This finishes the proof.
\end{proof}

To end this preliminary section, we prove the following tracial inequality.
\begin{lemma}\label{lemma:Powers-Stormer}
Let $B$ be a \cstar-algebra with $a,b \in B_+$ and $\tau \in T(B)$. Then
\[\|a-b\|_{2,\tau}^2 \leq \|a^2-b^2\|_{1,\tau}. \]
\end{lemma}
\begin{proof}
If we replace $B$ by its weak closure of the GNS representation $\pi_\tau(B)''$, it is enough to show this in case $B$ is a von Neumann algebra with faithful normal tracial state $\tau$. 

Historically, this was proved by Powers and St{\o}rmer in \cite[Lemma 4.1]{PowersStormer70} in case $B = M_n(\C)$ for some $n \in \N$.
When $B$ is a von Neumann algebra with faithful normal tracial state $\tau$ this follows from applying \cite[Lemma 2.10]{Haagerup75}, which is formulated for the space $L^2(B,\tau)$, to elements in $B \subset L^2(B,\tau)$.
For the reader's convenience we give here a more direct proof using an idea from \cite[Theorem 7.3.7]{DelarochePopa14}.

Given $a, b \in B_+$, let $p,q$ denote the spectral projections of $a-b$ corresponding to $[0,+\infty)$ and $(-\infty, 0)$, respectively.
This means that $a-b =(p-q)|a-b|$ and $p\perp q$.
First of all, we have
\begin{align*}
\tau((a^2-b^2)p) - \tau((a-b)^2p) &= \tau(b(a-b)p) + \tau((a-b)bp) \\
&= 2\tau(b^{1/2}(a-b)pb^{1/2}) \geq 0,
\end{align*}
since $(a-b)p \geq 0$. 
So we get 
\begin{equation}\label{eq:tracial_inequality_p}
\tau((a-b)^2p) \leq \tau((a^2-b^2)p),
\end{equation}
and in a similar way we can obtain that
\begin{equation}\label{eq:tracial_inequality_q}
\tau((b-a)^2q) \leq \tau((b^2-a^2)q).
\end{equation}
Combining equations \eqref{eq:tracial_inequality_p} and \eqref{eq:tracial_inequality_q} gives
\begin{equation*}
\tau((a-b)^2) = \tau((a-b)^2(p+q)) \leq \tau((a^2-b^2)(p-q)).
\end{equation*}
Also
\begin{align*}
\tau((a^2-b^2)(p-q))= \tau(p(a^2 - b^2)p) + \tau(q(b^2-a^2)q) \leq \tau(|a^2-b^2|(p+q)) \leq \|a^2-b^2\|_{1,\tau},
\end{align*}
since $\|p+q\| \leq 1$. This implies the result.
\end{proof}

\section{Equivariant uniform property Gamma}

The notion of uniform property Gamma was introduced in \cite{CETWW21} and further studied in \cite{CETW22}, where it served as a uniform \cstar-algebraic version of property Gamma introduced by Murray and von Neumann for II$_1$ factors \cite{MurrayvNeumann43}.
Recently, a dynamical version of this property was introduced in the separable unital setting in \cite{GardellaHirshbergVaccaro}, called \emph{equivariant uniform property Gamma}.
Here we revise the definition to account for separable \cstar-algebras with possibly unbounded traces, generalizing the concept called ``stabilised property Gamma'' by Castillejos--Evington \cite[Definition 2.5]{CastillejosEvington21}.
We choose not to adopt that name because one can argue that uniform property Gamma ought to be a stable property in the first place, just like property Gamma is for von Neumann algebras.
In light of recent work by Lin \cite{Lin22} who proposed a more general framework for \cstar-algebras that admit genuine quasitraces, we shall state the definition only in the absence of such.

For separable unital simple exact \cstar-algebras, the definition below corresponds to the earlier definition given in \cite{GardellaHirshbergVaccaro} (see Proposition \ref{prop:unital_Gamma_agrees_local_Gamma} below), but not in the non-simple case, as demonstrated by \cstar-algebras that arise as extensions of unital classifiable \cstar-algebras by the compacts (see the type of example mentioned in Remark \ref{rem:weird-examples-for-Gamma}, for instance).

\begin{definition}\label{definition:equivariantGamma}
Let $A$ be a separable \cstar-algebra with $Q\tilde{T}_2(A)=\tilde{T}(A)$ and $T^+(A) \neq \emptyset$, and let $\alpha:G \acts A$ be an action by a countable discrete group.
We say that $\alpha$ has \emph{equivariant uniform property Gamma} (or \emph{equivariant property Gamma} for short) if for all $n \in \N$, there exist pairwise orthogonal projections $p_1, \hdots, p_n \in (A^\omega\cap A')^{\alpha^\omega}$ such that for all $a \in A_+$ and $\tau \in \tilde{T}_\omega(A)$ with $\tau(a) < \infty$,
\[
\tau_a(p_i) = \frac{1}{n} \tau(a).
\]
\end{definition}

\begin{remark} \label{remark:Z-stability_implies_Gamma}
We can notice immediately from the naturality of the isomorphism in Proposition \ref{prop:tracial-central-sequence-algebra-stable} that equivariant uniform property Gamma is preserved under stable cocycle conjugacy.
That is, if $A$ and $B$ are \cstar-algebras as above and we have actions $\alpha: G\acts A$ and $\beta: G\acts B$ such that $\alpha\otimes\operatorname{id}_{\mathbb K}$ is cocycle conjugate to $\beta\otimes\operatorname{id}_{\mathbb K}$, then $\alpha^\omega$ is conjugate to $\beta^\omega$ via a map preserving the canonical traces.
In particular, equivariant uniform property Gamma holds for $\alpha$ if and only if it holds for $\beta$.

Next, we observe (cf.\ \cite[Proposition 2.3]{CETWW21}) that whenever $\alpha:G \acts A$ is an equivariantly $\mathcal{Z}$-stable action on a separable \cstar-algebra with $Q\tilde{T}_2(A)=\tilde{T}(A)$ and $T^+(A)\neq\emptyset$, it automatically has equivariant property Gamma.
Indeed, a cocycle conjugacy between $\alpha$ and $\alpha\otimes\operatorname{id}_\mathcal{Z}$ is easily seen to give rise to a unital $*$-homomorphism $\mathcal{Z}^\omega\cap \mathcal{Z}'\to (A^\omega\cap A')^{\alpha^\omega}$.
For this purpose one chooses an approximate unit $e_n\in A$ and considers a sequence of maps $\mathcal{Z}\to A\otimes\mathcal Z$, $x\mapsto e_n\otimes x$ composed with such a cocycle conjugacy, which is seen to induce such a homomorphism.
It is well-known that $\mathcal{Z}^\omega\cap \mathcal{Z}'$ admits unital embeddings of matrix algebras of arbitrary size $n\geq 2$.
So if we fix $n$ and define $p_1,\dots,p_n\in (A^\omega\cap A')^{\alpha^\omega}$ as the image of the canonical rank one projections inside a matrix algebra under the aforementioned $*$-homomorphism, then they satisfy the necessary requirements for equivariant property Gamma, by uniqueness of the trace on the $n\times n$ matrices.
\end{remark}

The following is a version of equivariant property Gamma for possibly nonseparable \cstar-algebras that exclusively takes into account the bounded traces. 
This agrees with \cite[Definition 3.1]{GardellaHirshbergVaccaro} for separable unital \cstar-algebras, but not with the general definition of equivariant property Gamma given above.

\begin{definition}\label{definition:localequivariantGamma}
Let $A$ be a $\sigma$-unital \cstar-algebra with $T(A)$ non-empty and compact, and let $\alpha: G \acts A$ be an action of a countable discrete group.
We say that $\alpha$ has \emph{local equivariant property Gamma with respect to bounded traces} if for all $n \in \N$ and ${\|\cdot\|_{2,T_\omega(A)}}$-separable subsets $S \subset A^{\omega,\rm b}$ there exist pairwise orthogonal projections $p_1, \hdots, p_n \in (A^{\omega,\rm b})^{\alpha^\omega} \cap S'$ such that $\tau(ap_i) = \frac{1}{n} \tau(a)$ for all $a \in S$ and $\tau \in T_\omega(A)$.
\end{definition}

In the unital separable simple setting, Definitions \ref{definition:equivariantGamma} and \ref{definition:localequivariantGamma} are equivalent.
We prove this fact in a slightly more general setting in the proposition below:

\begin{prop}\label{prop:unital_Gamma_agrees_local_Gamma}
Let $A$ be a simple separable \cstar-algebra with $Q\tilde{T}_2(A)=\tilde{T}(A)$ and such that $T(A)\neq\emptyset$ is compact and $T^+(A)=\R^{>0} T(A)$.\footnote{We note that this automatic if one assumes, e.g., that $A$ has continuous scale (see \cite[Definition 2.5]{Lin91}), which is a rather common assumption in the context of classification.}
Then an action $\alpha:G \acts A$ of a countable discrete group has equivariant property Gamma if and only if $\alpha$ has local equivariant property Gamma w.r.t.\ bounded traces.
\end{prop}
\begin{proof}
The assumptions on $A$ imply that every generalized limit trace on $A$ that is finite on some non-zero positive element of $A$ is a multiple of an ordinary limit trace (see Remark \ref{remark:generalized_limit_traces_bounded_case}), and that $A^\omega \cap A' = A^{\omega, \mathrm{b}} \cap A'$ (see Remark \ref{remark:consistent_notation_ultrapower}). Therefore, it suffices to show that the existence of pairwise orthogonal projections $p_1, \hdots, p_n \in (A^\omega \cap A')^{\alpha^\omega}$ such that $\tau(a p_i) = \frac{1}{n}\tau(a)$ for all $a \in A$ and $\tau \in T_\omega(A)$, implies for any $\|\cdot\|_{2,T_\omega(A)}$-separable $S \subset A^\omega$ the existence of pairwise orthogonal projections $p'_1, \hdots, p'_n \in (A^\omega \cap S')^{\alpha^\omega}$ such that $\tau(ap'_i) = \frac{1}{n}\tau(a)$ for all $a \in S$ and $\tau \in T_\omega(A)$. This follows by a standard reindexation argument, which we omit. 
\end{proof}

The next part of this section is devoted to proving an equivalence between equivariant property Gamma for an action $\alpha:G \acts A$ and local equivariant property Gamma w.r.t.\ bounded traces for its induced action $\alpha^\omega: G \acts A^\omega\cap A'$, at least in the setting when $A$ is simple nuclear and has stable rank one.\footnote{Although we use it in the proof, it is likely that stable rank one is not so important for the claim to hold, although we take no guess as to pinning down the correct general assumptions. We note however, that simple finite $\mathcal Z$-stable \cstar-algebras have stable rank one; see \cite{Rordam04srr,FuLiLin22}.}
Recall that $A$ is said to have stable rank one, if the invertibles of $\tilde{A}$ are dense in $\tilde{A}$.
We start by observing the following description of the tracial state space of $A^\omega\cap A'$. 

\begin{prop} \label{prop:tracesF(A)}
Let $A$ be a separable, simple, nuclear \cstar-algebra with uniform property Gamma and stable rank one.
Then every tracial state on $A^\omega\cap A'$ is a canonical trace, i.e., one has
\[
T(A^\omega\cap A') = \overbar{\mathrm{conv}}^{w^*}\{ \tau_a \mid \tau\in \tilde{T}_\omega(A),\ a \in A_+,\ \tau(a)=1 \}.
\]
\end{prop} 
\begin{proof}
Using exactly the same argument as in the proof of \cite[Lemma 3.3]{CastillejosEvington21} and modifying it as hinted in the remark stated before \cite[Theorem 3.4]{CastillejosEvington21}, we may appeal to \cite[Theorem 7.13]{AntoinePereraRobertThiel22} (since we assume stable rank one) and pick a non-zero hereditary \cstar-subalgebra $B \subset A \otimes \mathbb{K}$ with $T^+(B)=\mathbb{R}^{>0} T(B)$ and for which $T(B)$ is non-empty and compact.
By Brown's theorem, it follows that $A$ and $B$ are stably isomorphic.
Proposition \ref{prop:tracial-central-sequence-algebra-stable} implies that we have an isomorphism $A^\omega\cap A'\cong B^\omega\cap B'$ that induces a bijection between the canonical traces on the left and the right.
Hence the claim holds for $A$ if and only if it holds for $B$.

Now $B$ has uniform property Gamma (cf.\ Remark \ref{remark:Z-stability_implies_Gamma}), so \cite[Lemma 3.7]{CETWW21} implies that $B$ has CPoU.
By the `no silly trace' theorem \cite[Proposition 2.5]{CETW21}\footnote{Strictly speaking the conclusion is about the reduced tracial product $B^\infty$ in the reference, but this makes no difference to the argument there.}, one has that $T(B^\omega)$ is the weak-$*$-closed convex hull of the limit traces.
If $B$ is unital, then the claim follows directly from \cite[Proposition 4.6]{CETWW21}.
If $B$ is non-unital, we can extend the inclusion map $B\subset B^\omega$ to a unital inclusion $B^\dagger\subset B^\omega$.
From this point of view, we have a trivial equality of algebras
\[
B^\omega\cap B' = B^\omega\cap (B^\dagger)'\cap \{ 1_{B^\omega}-1_{B^\dagger} \}^\perp.
\]
In this case it follows from \cite[Proposition 5.7]{CastillejosEvington20} that $T(B^\omega\cap B')$ is the closed convex hull of traces of the form $\tau_{a}$, where $\tau\in T_\omega(B)$ is a limit trace and $a\in B^\dagger$ is a positive element with $\tau(a)=1$.
If $(e_n)_{n\in\N}$ is an increasing approximate unit in $B$, then $b_n=e_nae_n\in B$ converges to $a$ strictly, and hence $\|b_n-a\|_{2,\tau}\to 0$.
This implies the convergence of tracial states $\tau(b_n)^{-1}\tau_{b_n} \to \tau_a$ in the norm topology, so we observe the equality 
\[
T(B^\omega\cap B') = \overbar{\mathrm{conv}}^{w^*} \{ \tau_b \mid \tau\in T_\omega(B),\ b\in B_+,\ \tau(b)=1\},
\]
which proves the claim.
\end{proof}

\begin{theorem}\label{theorem:gamma_equivalence}
Let $A$ be a separable, simple, nuclear \cstar-algebra with stable rank one.
Then $\alpha: G \acts A$ has equivariant uniform property Gamma if and only if $\alpha^\omega: G \acts A^\omega\cap A'$ has local equivariant uniform property Gamma w.r.t.\ bounded traces.
\end{theorem}
\begin{proof}
In order to increase readability in this proof, let us specify another free ultrafilter $\kappa$ on $\mathbb N$ (which may or may not be equal to $\omega$).

We shall show the ``if'' part first, which actually holds for arbitrary separable simple \cstar-algebras with $Q\tilde{T}_2(A)=\tilde{T}(A)$ and $T^+(A)\neq\emptyset$.
Let $k\geq 2$.
Assuming $\alpha^\omega$ has local equivariant property Gamma w.r.t.\ bounded traces, we can find pairwise orthogonal projections $p_1, \hdots, p_k \in \big( (A^\omega\cap A')^{\kappa,\rm b}  \big)^{(\alpha^\omega)^\kappa}$ such that
\[
\tau(ap_j) = \frac{1}{k} \tau(a) \quad \text{ for } j=1,\dots,k,\ a \in A,\ \tau \in T_\kappa( A^\omega\cap A' ).
\]
For each $j=1,\hdots, k$, let $p_j$ be represented by a sequence of positive contractions $(p_{j,n})_{n \in \N}$ in $A^\omega\cap A'$.
Let in turn each element $p_{j,n}$ be represented by a central sequence $(x_{j,n,\ell})_{\ell \in \N}$ of positive contractions in $A$.
Traces in $T_\kappa(A^\omega\cap A')$ in particular include limit traces associated to sequences of canonical traces.
Let $C\subset\mathcal{P}(A)_+\setminus\{0\}$ be a countable dense subset.
Let $K\subset T^+(A)$ be a compact generator.
By the conclusion of Lemma \ref{lemma:compact_generator_no_cutdowns}, it follows for all $a\in C$ and all sequences $(\theta_\ell)_{\ell\in \N}$ in $K$ that, if $\tau$ is the limit trace on $A_\omega$ induced by $(\theta_\ell)_{\ell\in \N}$ and $\tau_a$ is the induced bounded trace on $A^\omega\cap A'$ that we view in a trivial way as a multiple of a (constant) limit trace on $(A^\omega\cap A')^{\kappa,\rm b}$, then
\[
0 =  \lim_{n\to\kappa} \|p_{j,n}-p_{j,n}^2\|_{1,\tau_a}
=  \lim_{n\to\kappa} \tau(a |p_{j,n}-p_{j,n}^2|)
\stackrel{\textup{Lemma } \ref{lemma:compact_generator_no_cutdowns}}{=} \lim_{n\to\kappa} \lim_{\ell\to\omega} \theta_\ell(a |x_{j,n,\ell}-x_{j,n,\ell}^2|).
\]
Since the sequence $(\theta_\ell)_{\ell \in \N}$ in $K$ was arbitrary, we may rewrite this as
\[
0 = \lim_{n\to\kappa} \lim_{\ell\to\omega} \max_{\theta\in K}\ \theta(a |x_{j,n,\ell}-x_{j,n,\ell}^2|).
\]
We may argue in a completely analogous fashion to see that
\[
0 = \lim_{n\to\kappa} \lim_{\ell\to\omega} \max_{\theta\in K}\ \theta(a |x_{j,n,\ell}-\alpha_g(x_{j,n,\ell})|),\quad g\in G,
\]
as well as
\[
0 = \lim_{n\to\kappa} \lim_{\ell\to\omega} \max_{\theta\in K}\ \big| \theta(a x_{j,n,\ell})-\frac1k\theta(a) \big| = \lim_{n\to\kappa} \lim_{\ell\to\omega} \max_{\theta\in K}\ \theta(a x_{j,n,\ell}x_{i,n,\ell})
\]
for all $i,j=1,\dots, k$ with $i\neq j$.
Lastly, we have by definition that $(x_{j,n,\ell})_{\ell \in \N}$ is a central sequence as $\ell\to\omega$.
Appealing to Kirchberg's $\varepsilon$-test, we can find central sequences of positive contractions $e^{(j)}_\ell$ in $A$ for $j=1,\dots,k$ satisfying for all $a\in C$ the properties
\[
0 = \lim_{\ell\to\omega} \max_{\theta\in K}\ \big| \theta(a e_{\ell}^{(j)})-\frac1k\theta(a) \big|  = \lim_{\ell\to\omega} \max_{\theta\in K}\ \theta(a |e_\ell^{(j)}-e_\ell^{(j)2}|)
\]
and
\[
0 = \lim_{\ell\to\omega} \max_{\theta\in K}\ \theta(a e_\ell^{(j)} e_{\ell}^{(i)}) = \lim_{\ell\to\omega} \max_{\theta\in K}\ \theta(a |e_\ell^{(j)}-\alpha_g(e_\ell^{(j)})|),\quad g\in G \text{ and } i\neq j.
\]
We consider the resulting elements $e_j\in A^\omega\cap A'$ represented by $(e_\ell^{(j)})_{\ell\in \N}$.
Given that $C$ was dense in $A_+$, we may conclude that they are pairwise orthogonal projections belonging to $(A^\omega\cap A')^{\alpha^\omega}$ satisfying $\tau_a(e_j)=\frac1k\tau(a)$ for all $\tau\in\tilde{T}_\omega(A)$ and $a\in C$ with $\tau(a)<\infty$.
In conclusion, this shows that $\alpha$ has equivariant uniform property Gamma.

For the ``only if'' part, suppose that $\alpha$ has equivariant property Gamma.
Given $k\geq 2$, there exist pairwise orthogonal projections $p_1, \hdots, p_k \in (A^\omega\cap A')^{\alpha^\omega}$ such that for all $a \in A_+$ and $\tilde{\tau} \in \tilde{T}_\omega(A)$ with $\tilde{\tau}(a) < \infty$
\[
\tilde{\tau}_a(p_j) = \frac{1}{k} \tilde{\tau}(a) \text{ for } j=1, \hdots, k.
\]
As above, choose a compact generator $K\subset T^+(A)$.
If we represent each element $p_j$ by a central sequence of positive contractions $(p_{j,n})_{n\in \N}$ in $A$, then we can argue as before and see that for all $a\in \mathcal{P}(A)_+\setminus\{0\}$, $g\in G$ and $i\neq j$, one has the limit properties
\[
0 = \lim_{n\to\omega} \max_{\theta\in K}\ \big| \theta(a p_{j,n})-\frac1k\theta(a) \big| = \lim_{n\to\omega} \max_{\theta\in K}\ \theta(a |p_{j,n}-\alpha_g(p_{j,n})|)
\]
and
\[
0 = \lim_{n\to\omega} \max_{\theta\in K}\ \theta(a p_{j,n} p_{i,n})  = \lim_{n\to\omega} \max_{\theta\in K}\ \theta(a |p_{j,n}-p_{j,n}^2|).
\]
Now take a countable subset $S \subset (A^\omega\cap A')^{\kappa,\rm b}$ whose closure would represent a separable subset as in Definition \ref{definition:localequivariantGamma}.
Without loss of generality, let us assume $S$ consists of positive elements.
Choose a countable subset $S_0\subset (A^\omega\cap A')_+$ such that every element of $S$ is represented by a bounded $S_0$-valued sequence.
Next, choose an increasing sequence of finite sets $F_n\subset \mathcal{P}(A)_+\setminus\{0\}$ such that their union is dense in $A_+$ and every element in $S_0$ has a representing sequence in $\prod_{n\in\N} F_n$.
Appealing to the above stated properties of the sequences $(p_{j,n})_{n\in \N}$ for $j=1,\dots,n$, we may find an increasing sequence of natural numbers $\ell\mapsto n_\ell$ such that the resulting subsequences satisfy
\begin{equation}\label{eq:subsequence_invariance}
0 = \lim_{\ell\to\infty} \max_{a\in F_\ell} \|[a,p_{j,n_\ell}]\|  =\lim_{\ell\to\infty} \max_{a\in F_\ell} \max_{\theta\in K}\ \theta(a |p_{j,n_\ell}-\alpha_g(p_{j,n_\ell})|) ,\end{equation}
\begin{equation}\label{eq:subsequence_orthogonal_projections} 0 = \lim_{\ell\to\infty} \max_{a\in F_\ell} \max_{\theta\in K}\ \theta(a p_{j,n_\ell} p_{i,n_\ell}) =\lim_{\ell\to\infty} \max_{a\in F_\ell} \max_{\theta\in K}\ \theta(a |p_{j,n_\ell}-p_{j,n_\ell}^2|),
\end{equation}
and
\begin{equation}\label{eq:subsequence_tracial_division}0  = \lim_{\ell\to\infty} \max_{a,b\in F_\ell} \max_{\theta\in K}\ \big| \theta(a b p_{j,n_\ell})-\frac1k\theta(ab) \big|
\end{equation}
for all $i,j=1,\dots,k$ with $i\neq j$.
By the choice of the sets $F_\ell$, we can see that $(p_{j,n_\ell})_{\ell\in \N}$ defines a central sequence in $A$, and its induced element $e_j\in A^\omega\cap A'$ commutes with elements in $S_0$.
We keep in mind the conclusion of Lemma \ref{lemma:compact_generator_no_cutdowns}.
Then conditions \eqref{eq:subsequence_invariance} and \eqref{eq:subsequence_orthogonal_projections} imply that $e_1,\dots,e_k$ are pairwise orthogonal projections in $(A^\omega\cap A')^{\alpha^\omega}$.
Condition \eqref{eq:subsequence_tracial_division} implies that for all $j=1,\dots,k$, $\tau\in\tilde{T}_\omega(A)$, every $a\in A_+$ with $\tau(a)=1$, and every $b\in S_0$, we have
\[
\tau_a(be_j)=\tau(abe_j)=\frac1k\tau(ab)=\frac1k\tau_a(b).
\]
By Proposition \ref{prop:tracesF(A)}, the weak-$*$-closed convex hull of such tracial states $\tau_a$ yields the whole tracial state space of $A^\omega\cap A'$.
In other words, we may conclude
\[
\tau(be_j)=\frac1k\tau(b) \quad\text{for all } j=1,\dots,k,\ b\in S_0 \text{ and } \tau\in T(A^\omega\cap A'). 
\]
We may view $e_j$ as constant elements inside $(A^\omega\cap A')^{\kappa,\rm b}$.
Since every element in $S$ was represented by a sequence in $S_0$, we may conclude that the elements $e_1,\dots,e_k$ satisfy the required property from Definition \ref{definition:localequivariantGamma} applied to the action $\alpha^\omega: G\acts A^\omega\cap A'$.
We conclude that $\alpha^\omega$ has local equivariant property Gamma w.r.t.\ bounded traces.
\end{proof}

\section{Dynamical complemented partitions of unity}

This section contains the most involved technical arguments of the article, namely the proof that local equivariant property Gamma implies the existence of a dynamical version of complemented partitions of unity \cite[Definition 3.1]{CETWW21}, or dynamical CPoU for short.
In the case where the induced action on the tracial state space has the property that all orbits are finite with uniformly bounded cardinality, a different iteration of dynamical CPoU was proved in \cite[Theorem 4.3]{GardellaHirshbergVaccaro}.
However, we note that the general statement we prove is a weaker and more intricate version compared to earlier versions, but will nevertheless be sufficient to deduce the tracial local-to-global principle.

The starting point for the approach in this section is the following weaker version of CPoU shown in \cite[Lemma 3.6]{CETWW21} for nuclear \cstar-algebras, which turns out to hold automatically with the aid of the theory of tracially complete \cstar-algebras \cite{CCEGSTW}.

\begin{prop} \label{prop:weak-cpou}
Let $A$ be a $\sigma$-unital \cstar-algebra with $T(A)$ non-empty and compact.
Then for every $\|\cdot\|_{2,T_\omega(A)}$-separable subset $S \subset A^{\omega,\rm b}$, every $k\in \N$, every family $a_1, \hdots, a_k \in A_+$ and every 
\[
\delta > \sup_{\tau \in T(A)} \min_{i=1, \hdots,  k} \tau (a_i),
\]
there exist $e_1, \hdots, e_k \in (A^{\omega,\rm b} \cap S')_+^1$ such that for all $\tau \in T_\omega(A)$
\begin{itemize}[itemsep=1ex,topsep=1ex]
\item $\tau(\sum_{i=1}^k e_i)= 1$,
\item $\tau(a_ie_i) \leq \delta\tau(e_i)$ for $i=1, \hdots, k$. 
\end{itemize}
\end{prop}
\begin{proof}
Let $S$, $k$, $a_1,\hdots, a_k$ and $\delta$ be chosen as in the assumption.
Set 
\[
\delta_0:=\sup_{\tau \in T(A)} \min_{i=1, \hdots,  k} \tau (a_i) < \delta.
\]
Since $S$ is $\|\cdot\|_{2,T_\omega(A)}$-separable, it is first of all clear that one may find a non-degenerate separable \cstar-subalgebra $A_0\subseteq A$ containing the tuple $a_1,\dots, a_k$ such that every element of $S$ can be represented by a bounded sequence in $A_0$.
As every tracial state on $A$ restricts to one on $A_0$, the tracial state space of $A_0$ is still non-empty and compact, and furthermore
\[
\sup_{\tau \in T(A_0)} \min_{i=1, \hdots,  k} \tau (a_i) \geq \delta_0.
\]
Let $\eta>0$.
We claim that there exists a finite set $F_\eta\ssubset A$ and $\varepsilon_\eta>0$ such that if $\rho$ is any state on $A$ with 
\[
\max_{x\in F_\eta} |\rho(x^*x)-\rho(xx^*)|<\varepsilon_\eta,
\]
then $\min_{i=1,\hdots,k} \rho(a_i)<\delta_0+\eta$.
If we suppose for a moment that this were false, then it follows that for every finite set $F\ssubset A$ and every $\varepsilon>0$ there exists a state $\rho_{(F,\varepsilon)}$ on $A$ with
\[
\max_{x\in F} |\rho_{(F,\varepsilon)}(x^*x)-\rho_{(F,\varepsilon)}(xx^*)|<\varepsilon \quad\text{and}\quad \min_{i=1,\hdots,k} \rho_{(F,\varepsilon)}(a_i)\geq \delta_0+\eta.
\]
We can view $\rho_{(F,\varepsilon)}$ as a net of states by equipping the set of pairs $(F,\varepsilon)$ with the obvious order.
By the Banach--Anaoglu theorem, there exists a subset $(\rho_\lambda)_{\lambda\in\Lambda}$ that weak-$*$-converges to a positive functional $\rho'$ with norm at most one on $A$.
By the properties of the net $\rho_{(F,\varepsilon)}$, it is clear that $\rho'$ is tracial.
Hence $\min_{i=1,\hdots,k} \rho'(a_i)\leq\delta_0$, while at the same time
\[
\min_{i=1,\hdots,k} \rho'(a_i) = \lim_{(F,\varepsilon)} \min_{i=1,\hdots,k} \rho_{(F,\varepsilon)}(a_i) \geq \delta_0+\eta,
\]
which is a contradiction.

Using this intermediate claim, we choose for each $n\geq 1$ a finite set $F_n\ssubset A$ and $\varepsilon_n>0$ satisfying the above conclusion for $\eta=\frac1n$.
Let $A_1\subseteq A$ be the \cstar-algebra generated by $A_0$ and all the finite sets $F_n$, which is clearly still separable.
Since $A_1$ contains all the finite sets $F_n$, it follows that every tracial state $\tau$ on $A_1$ must satisfy
\[
\min_{i=1,\hdots,k} \tau(a_i) \leq \delta_0 + \frac1n,\quad n\geq 1,
\]
which leads to
\[
\sup_{\tau \in T(A_1)} \min_{i=1, \hdots,  k} \tau (a_i) = \delta_0 < \delta.
\]
By all the properties arranged for the subalgebra $A_1\subseteq A$ so far, it is clear for proving our main claim that we may swap $A$ for the subalgebra $A_1$.
In other words, we may assume without loss of generality that $A$ is separable.

By \cite[Definition 3.19, Proposition 3.23]{CCEGSTW}, the tracial completion $\overline{A}^{T(A)}$ of $A$ yields a factorial tracially complete \cstar-algebra.
Note that as per the ultraproduct construction of tracially complete \cstar-algebras in \cite{CCEGSTW}, the object $\big(\overline{A}^{T(A)})^\omega$ in that sense becomes canonically isomorphic to the \cstar-algebra $A^{\omega,\rm b}$ as considered in Definition~\ref{definition:tracial_ultrapowers}.
Because $A$ is separable, $\overline{A}^{T(A)}$ is $\|\cdot\|_{2,T(A)}$-separable.
Thus we may directly apply \cite[Theorem 6.15]{CCEGSTW} (inserting the unit in place of the projection $q$ appearing there) and find the elements $e_1,\dots,e_k\in (A^{\omega,\rm b} \cap S')_+^1$ with the desired properties.
\end{proof}

The main achievement of this section is the following technical lemma:

\begin{lemma} \label{lemma:equivariant_CPOU}
Given $\e >0$ and $t \in (0,1)$, there exists a universal constant $\eta=\eta(\e,t) >0$ such that the following holds:
Let $A$ be a $\sigma$-unital \cstar-algebra with $T(A)$ non-empty and compact.
Let $G$ be a countable discrete group, and let $\alpha: G \acts A$ be an action with local equivariant property Gamma w.r.t.\ bounded traces.
Suppose that $F,H \ssubset G$ are finite subsets such that $|gH \Delta H| < \eta|H|$ for all $g\in F$.
Then for every $\|\cdot\|_{2,T_\omega(A)}$-separable subset $S \subset A^{\omega,\rm b}$, every family $a_1, \hdots, a_k \in (A^{\omega,\rm b})_+$, and every constant $\delta>0$ with
\begin{equation}\label{eq:conditions_a_ultrapower}
\frac{\delta}{|H|} > \sup_{\tau \in T_\omega(A)}\min_{i = 1, \hdots, k} \tau(a_i),
\end{equation}
there exist pairwise orthogonal projections $p_1, \hdots, p_k \in A^{\omega,\rm b} \cap S'$ such that for all $\tau \in T_\omega(A)$ one has
\begin{align}
\tau(p_1 + \hdots + p_k) &> t,\\
\tau(a_ip_i) &\leq \delta \tau(p_i) \text { for } i=1, \hdots, k, \text{ and}\label{eq:delta_inequality_theorem}\\
\textstyle \max_{g \in F} \sum_{i=1}^k\|\alpha^\omega_g(p_i)-p_i\|_{2,\tau}^2 &<\e.\label{eq:sum_invariance_theorem}
\end{align}
\end{lemma}
\begin{remark}\label{remark:reduction_constant_elements}
A standard argument shows that the statement in Lemma \ref{lemma:equivariant_CPOU} is equivalent to the existence of a universal constant $\eta(\e,t) > 0$ satisfying the following statement (using approximations instead of the uniform bounded tracial ultrapower):

\emph{If $\alpha: G \acts A$ is an action and $F,H \ssubset G$ are all given as in Lemma \ref{lemma:equivariant_CPOU}, then for every finite subset $S \ssubset A$, every $\xi >0$, every family $a_1, \hdots, a_k \in A_+$, and every $\delta >0$ with
\begin{equation}\label{eq:conditions_a_approximate}
\frac{\delta}{|H|} > \sup_{\tau \in T(A)}\min_{i = 1, \hdots, k} \tau(a_i),
\end{equation}
there exist pairwise orthogonal contractions $e_1, \hdots, e_k \in A_+$ such that 
\begin{align*}
\|[e_i,x]\|_{2,u} &< \xi \text{ for } x \in S, i=1, \hdots, k\\
\|e_i-e_i^2\|_{2,u} &< \xi \text{ for } i =1, \hdots, k,\\
\tau(e_1 + \hdots + e_k) &> t - \xi \text{ for } \tau \in T(A)\\
\tau(a_ie_i) &< \delta \tau(e_i) + \xi \text{ for } \tau \in T(A), i=1, \hdots, k, \text{ and}\\
\textstyle \max_{g \in F} \sum_{i=1}^k \|\alpha_g(e_i) - e_i\|_{2,\tau} &< \e + \xi \text{ for } \tau \in T(A).
\end{align*}}

In particular, this means that it suffices to prove Lemma \ref{lemma:equivariant_CPOU} for positive elements $a_1, \hdots, a_k$ taken in $A$ instead of $A^{\omega,\rm b}$. In this case, \eqref{eq:conditions_a_ultrapower} and \eqref{eq:conditions_a_approximate} are equivalent. 
\end{remark}

The proof of Lemma \ref{lemma:equivariant_CPOU} is an adapted version of the proof in the non-dynamical setting (cf.\ \cite[Section 3]{CETWW21}), but also incorporates new ideas related to the dynamical structure.
Before we delve into the details, we shall give an overview of the strategy.
The construction of the pairwise orthogonal projections $p_1, \hdots, p_k$ in the statement of Lemma \ref{lemma:equivariant_CPOU} is done in three steps:

\begin{enumerate}
\item Instead of producing pairwise orthogonal projections $p_1, \hdots, p_k$, we start by producing (not yet pairwise orthogonal) positive contractions $e_1, \hdots, e_k\in A^{\omega,\rm b} \cap S'$ that satisfy
\[
\tau(e_1 + \hdots + e_k) = 1,\ \tau(a_ie_i) \leq \delta \tau(e_i) \quad \text{for } i=1, \hdots, k, \, \tau \in T_\omega(A),
\]
and that are approximately invariant under $\alpha^\omega$ in the right sense. This is done in Lemma \ref{lemma:precursor_1} and is the only part of the proof that makes use of the approximate F{\o}lner property that appears in the assumption of the lemma.
\item Next, we use equivariant property Gamma to turn these contractions into orthogonal projections $p_1', \hdots, p'_k \in A^{\omega,\rm b} \cap S'$.
As a consequence of this procedure we get that 
\[
\tau(p_1' + \hdots + p'_k) =\frac{1}{k} \quad \text{ for } \tau \in T_\omega(A),
\]
but they still satisfy \eqref{eq:delta_inequality_theorem} and are still approximately invariant in the right sense.
This is done in Lemma \ref{lemma:precursor_2}.
\item In order to enlarge the trace of the sum of the projections, we repeat the above steps underneath the projection $1_{A^\omega} - \sum_{i=1}^k p'_i$.\footnote{For this one actually needs a somewhat stronger version of the second step, see Lemma \ref{lemma:precursor_final}.}
We continue this procedure inductively until we end up with orthogonal projections $p_1, \hdots, p_k$ whose sum exceeds $t$ in trace and that still satisfy \eqref{eq:delta_inequality_theorem}.
If everything is done carefully from the start and $\eta>0$ is chosen correctly, we can control the error in the invariance of the projections and make sure they satisfy \eqref{eq:sum_invariance_theorem} in the end.
(We note, informally, that this error grows with the number of times this procedure is repeated, which is the ultimate reason why we cannot simply work with $t=1$ in the statement.)
\end{enumerate}

We shall now implement the above strategy.
Combining the contractions arising from Proposition \ref{prop:weak-cpou} with an averaging argument over suitable F\o lner sets allows us to carry out the first step:

\begin{lemma} \label{lemma:precursor_1}
Let $A$ be a $\sigma$-unital \cstar-algebra with $T(A)$ non-empty and compact.
Let $\alpha:G \acts A$ be an action by a countable discrete group.
Let $\e >0$ and finite subsets $F \ssubset G$ and $H \ssubset G$ be given such that $|gH\Delta H| < \e |H|$ for all $g \in F$.
Then for all $\|\cdot\|_{2,T_\omega(A)}$-separable subsets $S \subset A^{\omega,\rm b}$, all $\delta >0$ and all $a_1, \hdots, a_k \in A_+$ with  
\[
\frac{\delta}{|H|} > \sup_{\tau \in T(A)}\min_{i = 1, \hdots, k}\ \tau(a_i),
\]
there exist $e_1, \hdots, e_k \in (A^{\omega,\rm b} \cap S')_+^1$ such that for $\tau \in T_\omega(A)$:
\begin{itemize}[itemsep=1ex,topsep=1ex]
\item $\tau(\sum_{i=1}^k e_i)=1$,
\item $\tau(a_ie_i) \leq \delta \tau(e_i)$ for $i=1, \hdots, k$, and
\item $\displaystyle \max_{g \in F} \sum_{i=1}^k\|\alpha^\omega_g(e_i) - e_i\|_{1,\tau} < \e$.
\end{itemize}                                                                                                                                                                                                                                                                                                            
\end{lemma}
\begin{proof}
Given $S \subset A^{\omega,\rm b}$, $\delta >0$, and  $a_1, \hdots, a_k \in A_+$ as above, we define
\[
a_i' := \frac{1}{|H|} \sum_{g \in H} \alpha_{g^{-1}}(a_i) ,\quad i=1,\dots,k.
\]
Note that for each $\tau \in T(A)$ the trace $\frac{1}{|H|} \sum_{g \in H} \tau \circ \alpha_{g^{-1}}$ is again an element of $T(A)$, so we see that
\[ \frac{\delta}{|H|} > \sup_{\tau \in T(A)}\min_{i = 1, \hdots, k}\ \tau(a_i').\]
By Proposition \ref{prop:weak-cpou}, we know that there exist $e'_1, \hdots, e'_k \in ( A^\omega\cap (\bigcup_{g \in G} \alpha_g^\omega(S))')_+^1$ such that for $\tau \in T_\omega(A)$
\begin{align}\tau(\sum_{i=1}^k e'_i) &=1 \text{ and }\label{eq:first_precursor_sum_1}\\
 \tau(a_i'e'_i) &\leq \frac{\delta}{|H|}\tau(e'_i) \quad \text{ for } i=1, \hdots, k.\nonumber
\end{align}
In particular, this last equation implies that
\begin{equation}\label{eq:first_precursor_delta_inequality} 
\tau(\alpha_{g^{-1}}(a_i)e_i') \leq \delta \tau(e_i') ,\quad g\in H,\, \tau \in T_\omega(A).
\end{equation}
Now for $i = 1, \hdots, k$, define
\[
e_i :=|H|^{-1} \sum_{g \in H} \alpha_g^\omega(e_i').
\]
Clearly this still is a positive contraction in $A^\omega \cap S'$. Notice that 
\[
\tau\Big(\sum_{i=1}^ke_i\Big) =|H|^{-1} \sum_{g \in H}  (\tau \circ \alpha_g^\omega)(\sum_{i=1}^k e'_i)  \overset{\eqref{eq:first_precursor_sum_1}}{=} 1 \quad \text{for }\tau \in T_\omega(A).
\]
For $\tau \in T_\omega(A)$ and $i=1, \hdots, k$ we have
\begin{align*} \tau(a_ie_i) &=|H|^{-1} \tau(a_i \displaystyle\sum_{g \in H} \alpha_g^\omega(e'_i))\\
&= |H|^{-1} \displaystyle\sum_{g \in H}(\tau \circ \alpha_g^\omega)(\alpha_{g^{-1}}(a_i)e'_i)\\
\overset{\eqref{eq:first_precursor_delta_inequality}}&{\leq} |H|^{-1} \displaystyle\sum_{g \in H}\delta(\tau \circ \alpha_g^\omega) (e'_i)\\
&= \delta \tau(e_i).\\ 
\end{align*}
Lastly, we see that for $g \in F$ and $\tau \in T_\omega(A)$ we have
\begin{align*}
\sum_{i=1}^k\|\alpha_g^\omega(e_i) - e_i \|_{1,\tau} &\leq |H|^{-1}\sum_{i=1}^k \sum_{h \in gH \Delta H}\|\alpha_h^\omega(e_i')\|_{1,\tau}\\
&= |H|^{-1} \sum_{h \in gH \Delta H}\sum_{i=1}^k\tau(\alpha_h^\omega(e_i'))\\
\overset{\eqref{eq:first_precursor_sum_1}}&{=}|H|^{-1}|gH\Delta H| < \e.
\end{align*}
\end{proof}

Analogously as in the non-dynamical setting (cf.\ \cite[Lemma 2.4]{CETWW21}), (local) equivariant property Gamma allows one to replace positive contractions by projections without changing the tracial values in $A^{\omega,\rm b}$.
A different generalization of this lemma was proved in \cite[Proposition 3.4]{GardellaHirshbergVaccaro}, but for the purposes of this paper we need a way to control the (tracially) approximate fixedness of the elements for the action.

\begin{lemma}
\label{lemma:eq_tracial_proj}
Let $A$ be a $\sigma$-unital \cstar-algebra with $T(A)$ non-empty and compact, and let $\alpha: G \acts A$ be an action of a countable discrete group.
Assume that $\alpha$ has local equivariant property Gamma w.r.t.\ bounded traces.
Let $S \subset A^{\omega,\rm b}$ be a $\|\cdot\|_{2,T_\omega(A)}$-separable subset, and let $b \in A^{\omega,\rm b}\cap S'$ be a positive contraction.
Then there exists a projection $p \in A^{\omega,\rm b}\cap S'$ such that 
\begin{equation}\label{eq:same_tracial_behavior}
\tau(ap) = \tau(ab)\quad \text{ for } a \in S,\ \tau \in T_\omega(A),\end{equation}
and such that for all $g \in G$ and $\tau \in T_\omega(A)$ one has
\begin{equation}\label{eq:tracial_inequality}
\|\alpha^\omega_g(p) - p\|^2_{2,\tau} \leq \|\alpha^\omega_g(b)^{1/2}- b^{1/2}\|_{2,\tau}\|\alpha^\omega_g(b)^{1/2}+ b^{1/2}\|_{2,\tau}.
\end{equation}
\end{lemma}
\begin{proof}
Fix $n \in \N$.
By a common reindexation trick, it suffices to find a positive contraction $p\in A^{\omega,\rm b}\cap S'$ satisfying \eqref{eq:same_tracial_behavior}, \eqref{eq:tracial_inequality} and $\|p-p^2\|^2_{2,T_\omega(A)} \leq 1/n.$
An element $p$ satisfying all the necessary properties except \eqref{eq:tracial_inequality} is constructed in the proof of \cite[Lemma 2.4]{CETWW21} with the use of uniform property Gamma.
We show that when the construction is carried out using (local) equivariant property Gamma instead, the resulting projection also satisfies the extra condition \eqref{eq:tracial_inequality}.

Just as in \cite{CETWW21}, we define functions $f_1, \hdots, f_n \in C([0,1])$ such that ${f_i}\big\rvert_{ [0,(i-1)/n]} = 0, {f_i}\big\rvert_{ [i/n,1]} = 1,$ and $f_i$ is linear on $[(i-1)/n,i/n]$.  Note that not only $\frac{1}{n}\sum_{i=1}^n f_i = \mathrm{id}_{[0,1]}$, but the monotonicity of each $f_i$ also implies
\begin{equation}\label{eq:sum_unity}
\frac{1}{n}\sum_{i=1}^n |f_i(t_1)-f_i(t_2)| = |t_1-t_2|,\quad t_1,t_2\in [0,1]. 
\end{equation}
By local equivariant property Gamma w.r.t.\ bounded traces we can find pairwise orthogonal projections $p_1, \hdots, p_n \in (A^{\omega,\rm b})^{\alpha^\omega} \cap S' \cap \{b\}'$ such that 
$\tau(p_ix) = \frac{1}{n} \tau(x)$ for $i=1, \hdots, n, \tau \in T_\omega(A), x \in C^*(S \cup\{b\})$. Define 
\[
p:= \sum_{i = 1}^n p_i f_i(b) \in A^{\omega,\mathrm{b}} \cap S'.
\]
By repeating the arguments in the proof of \cite[Lemma 2.4]{CETWW21} verbatim, we may conclude that $\|p-p^2\|^2_{2,T_\omega(A)} \leq \frac{1}{n}$ and $\tau(ap) = \tau(ab)$ for all $a \in S$ and $\tau \in T_\omega(A)$. 
We need to show that \eqref{eq:tracial_inequality} holds as well.
Fix $\tau \in T_\omega(A)$ and $g\in G$. By \cite[I.1]{Connes76} there exists a positive Radon measure $\nu$ on $[0,1]^2$ such that for every pair of functions $h_1,h_2 \in C_0((0,1])$, the functions $(s, t) \mapsto h_1(s)$ and $(s, t) \mapsto h_2(t)$ are square integrable, and
\[
\|h_1(\alpha^\omega_g(b)) - h_2(b)\|_{2,\tau}^2 = \int_{[0,1]^2}|h_1(s)-h_2(t)|^2\, \mathrm{d}\nu(s,t).
\]
Then we get
\begin{align*}
\|\alpha^\omega_g(p)-p\|_{2,\tau}^2 &= \big\| \sum_{i=1}^n p_i (f_i(\alpha^\omega_g(b))- f_i(b)) \big\|_{2,\tau}^2 \\
&= \frac{1}{n} \sum_{i=1}^n \| f_i(\alpha^\omega_g(b))- f_i(b)\|_{2,\tau}^2\\
&= \frac{1}{n} \sum_{i=1}^n \int_{[0,1]^2} |f_i(s) - f_i(t)|^2 \, \mathrm{d} \nu(s,t)\\
&\leq \frac{1}{n} \sum_{i=1}^n \int_{[0,1]^2} |f_i(s) - f_i(t)|\, \mathrm{d} \nu(s,t)\\
\overset{\eqref{eq:sum_unity}}&{=} \int_{[0,1]^2} |s-t|\, \mathrm{d} \nu(s,t)\\
&\leq \sqrt{\int_{[0,1]^2} |s^{1/2}- t^{1/2}|^2 \, \mathrm{d} \nu(s,t)}\sqrt{\int_{[0,1]^2} |s^{1/2}+ t^{1/2}|^2 \, \mathrm{d} \nu(s,t)}\\
&= \|\alpha^\omega_g(b)^{1/2}- b^{1/2}\|_{2,\tau}\|\alpha^\omega_g(b)^{1/2}+ b^{1/2}\|_{2,\tau}.
\end{align*} 
This ends the proof.
\end{proof}

Using Lemma \ref{lemma:eq_tracial_proj}, we can construct orthogonal projections that play a similar role to the positive elements in Lemma \ref{lemma:precursor_1}.

\begin{lemma}\label{lemma:precursor_2}
Let $A$ be a $\sigma$-unital \cstar-algebra with $T(A)$ non-empty and compact.
Let $\alpha:G \acts A$ be an action by a countable discrete group and assume it has local equivariant property Gamma w.r.t.\ bounded traces.
Let $\e >0$, $F \ssubset G$ and $H \ssubset G$ be such that $|gH \Delta H| < \e|H|$ for all $g \in F$.
Then for every $\|\cdot\|_{2,T_\omega(A)}$-separable subset $S \subset A^{\omega,\rm b}$, all $\delta >0$ and all $a_1, \hdots, a_k \in A_+$ with  
\[
\frac{\delta}{|H|} > \sup_{\tau \in T(A)}\min_{i = 1, \hdots, k}\ \tau(a_i),
\]
there exist pairwise orthogonal projections $p_1, \hdots, p_k \in  A^{\omega,\mathrm{b}} \cap S'$ such that for all $\tau \in T_\omega(A)$ 
\begin{itemize}[itemsep=1ex,topsep=1ex]
\item $\tau(\sum_{i=1}^k p_i) = \frac{1}{k}$,
\item $\tau(a_ip_i) \leq \delta \tau(p_i)$ for $i=1, \hdots, k$, and
\item $\displaystyle \max_{g \in F} \sum_{i=1}^k\|\alpha^\omega_g(p_i) -p_i\|_{2,\tau}^2 < \frac{2 \sqrt{\e}}{k}$ . 
\end{itemize}
\end{lemma}
\begin{proof}
By Lemma \ref{lemma:precursor_1}, we can find $e_1, \hdots, e_k \in (A^{\omega,\rm b} \cap S')_+^1$ such that for all $\tau \in T_\omega(A)$
\begin{align}
\tau(\sum_{i=1}^k e_i)&=1, \label{eq:precursor_unity}\\
\tau(a_ie_i) &\leq \delta \tau(e_i) \text{ for } i=1, \hdots, k, \text{ and}\label{eq:precursor_delta_inequality}\\
\max_{g \in F} \sum_{i=1}^k\|\alpha^\omega_g(e_i) - e_i\|_{1,\tau} &< \e.\label{eq:precursor_invariance}
\end{align}
Let $S_0 = S \cup \{1_{A^{\omega,\rm b}}, a_1, \hdots, a_k\}$. Apply Lemma \ref{lemma:eq_tracial_proj} for each $i \in \{1,\hdots, k\}$ and find a projection $p_i \in A^{\omega,\mathrm{b}} \cap S_0'$ such that for all $a \in S_0$, $\tau \in T_\omega(A)$ and $g\in G$, we have
\begin{align}
 \tau(ap_i) &= \tau(ae_i), \label{eq:precursor_same_tracial_behavior}\\
 \|\alpha^\omega_g(p_i) - p_i\|^2_{2,\tau} &\leq \|\alpha^\omega_g(e_i)^{1/2}- e_i^{1/2}\|_{2,\tau}\|\alpha^\omega_g(e_i)^{1/2}+ e_i^{1/2}\|_{2,\tau}.\label{eq: precursor_invariance_relation}
\end{align}
This already implies the two following facts:
\begin{align}
&\tau(\sum_{i=1}^k p_i) \overset{\eqref{eq:precursor_same_tracial_behavior}}{=} \tau(\sum_{i=1}^k e_i) \overset{\eqref{eq:precursor_unity}}{=} 1 \text{ for }\tau \in T_\omega(A), \text{ and} \label{eq:precursor_sum_1_projections}\\
& \tau(a_ip_i)\overset{\eqref{eq:precursor_same_tracial_behavior}}{=}\tau(a_ie_i) \overset{\eqref{eq:precursor_delta_inequality}}{\leq}\delta \tau(e_i) \overset{\eqref{eq:precursor_same_tracial_behavior}}{=} \delta \tau(p_i) \text{ for } i=1, \hdots, k, \,\tau \in T_\omega(A). \label{eq:precursor_delta_inequality_projection}\end{align}
Furthermore, we get for $g \in F$ and $\tau \in T_\omega(A)$ that
\begin{align}
\sum_{i=1}^k\|\alpha^\omega_g(p_i) -p_i\|_{2,\tau}^2  \overset{\eqref{eq: precursor_invariance_relation}}&{\leq} \sum_{i=1}^k \|\alpha^\omega_g(e_i)^{1/2}- e_i^{1/2}\|_{2,\tau}\|\alpha^\omega_g(e_i)^{1/2}+ e_i^{1/2}\|_{2,\tau}\nonumber\\
&\leq \sum_{i=1}^k \|\alpha^\omega_g(e_i)^{1/2}- e_i^{1/2}\|_{2,\tau}(\|\alpha_g^\omega(e_i)^{1/2}\|_{2,\tau}+ \|e_i^{1/2}\|_{2,\tau})\nonumber\\
 \overset{\text{Lemma } \ref{lemma:Powers-Stormer}}&{\leq}\sum_{i=1}^k \|\alpha^\omega_g(e_i)- e_i\|_{1,\tau}^{1/2}(\|\alpha_g^\omega(e_i)^{1/2}\|_{2,\tau}+ \|e_i^{1/2}\|_{2,\tau})\nonumber\\
&= \sum_{i=1}^k \|\alpha^\omega_g(e_i)- e_i\|_{1,\tau}^{1/2} \tau( \alpha_g^\omega(e_i))^{1/2}+ \sum_{i=1}^k \|\alpha^\omega_g(e_i)- e_i\|_{1,\tau}^{1/2}\tau(e_i)^{1/2} \nonumber\\
& \leq \sqrt{\sum_{i=1}^k \|\alpha^\omega_g(e_i)- e_i\|_{1,\tau} \sum_{i=1}^k (\tau\circ \alpha_g^\omega)(e_i)}+\sqrt{\sum_{i=1}^k \|\alpha^\omega_g(e_i)- e_i\|_{1,\tau} \sum_{i=1}^k \tau(e_i)}\nonumber\\
\overset{\eqref{eq:precursor_unity},\eqref{eq:precursor_invariance}}&{<} 2 \sqrt{\e}.\label{eq:precursor_invariance_projections}
\end{align}
Set 
\[
S_1 = S \cup \mathrm{C}^*\Big( \{a_1, \hdots, a_k\} \cup \{ \alpha^\omega_g(p_j) \mid 1\leq j\leq k,\ g\in G\} \Big) \subset A^{\omega,\mathrm{b}}.
\]
Since $A$ has local equivariant property Gamma w.r.t.\ bounded traces we can find pairwise orthogonal projections $r_1, \hdots, r_k \in (A^{\omega,\rm b} )^{\alpha^\omega}\cap S_1'$ such that
\begin{equation}\label{eq:precursor_equivariant_gamma}
\tau(r_ia) = \frac{1}{k}\tau(a) \quad \text{for } \tau \in T_\omega(A),\ a \in S_1,\ i=1, \hdots, k.
\end{equation}
Set $p_i' := r_ip_i$.
Then clearly $p_1', \hdots, p_k' \in A^{\omega,\rm b} \cap S'$ are pairwise orthogonal projections.
We get for each $\tau \in T_\omega(A)$ that
\[
\tau(\sum_{i=1}^k p'_i) = \tau(\sum_{i=1}^k r_ip_i) \overset{\eqref{eq:precursor_equivariant_gamma}}{=} \frac{1}{k}\tau(\sum_{i=1}^k p_i) \overset{\eqref{eq:precursor_sum_1_projections}}{=} \frac{1}{k}.
\]
For $i=1, \hdots, k$ and $\tau \in T_\omega(A)$ we get
\[\tau(a_ip'_i) = \tau(a_ir_ip_i) \overset{\eqref{eq:precursor_equivariant_gamma}}{=} \frac{1}{k}\tau(a_ip_i) \overset{\eqref{eq:precursor_delta_inequality_projection}}{\leq}\frac{\delta}{k}\tau(p_i) \overset{\eqref{eq:precursor_equivariant_gamma}}{=} \delta \tau(p_i').\]
Lastly, for $g \in F$ and $\tau \in T_\omega(A)$ we get
\[\sum_{i=1}^k\|\alpha^\omega_g(p'_i) -p'_i\|_{2,\tau}^2 \overset{\eqref{eq:precursor_equivariant_gamma}}{=} \frac{1}{k}\sum_{i=1}^k\|\alpha^\omega_g(p_i) -p_i\|_{2,\tau}^2 \overset{\eqref{eq:precursor_invariance_projections}}{<} \frac{2\sqrt{\e}}{k}.\]
Hence the projections $p_i'$ satisfy all the required properties.
This finishes the proof.
\end{proof}

In order to carry out the inductive argument to enlarge the trace of the sum of the constructed orthogonal projections, we need a stronger version of the previous lemma that allows us to find the orthogonal projections under an arbitrary tracially constant projection $q$:

\begin{lemma}\label{lemma:precursor_final}
Let $A$ be a $\sigma$-unital \cstar-algebra with $T(A)$ non-empty and compact.
Let $\alpha:G \acts A$ be an action by a countable discrete group and assume it has local equivariant property Gamma w.r.t.\ bounded traces.
Let $\e >0$, $F \ssubset G$ and $H \ssubset G$ be such that $|gH \Delta H| < \e|H|$ for all $g \in F$.
Choose $\delta >0$, and  $a_1, \hdots, a_k \in A_+$ such that
\[
\frac{\delta}{|H|} > \sup_{\tau \in T(A)}\min_{i = 1, \hdots, k}\ \tau(a_i).
\]
For every $\mu \in (0,1]$ and $\|\cdot\|_{2,T_\omega(A)}$-separable subset $S_0 \subset A^{\omega,\rm b}$, there exists a $\|\cdot\|_{2,T_\omega(A)}$-separable subset $S \subset A^{\omega,\rm b}$ such that if $q \in A^{\omega,\rm b} \cap S'$ is a projection with $\tau(q)=\mu$ for all $\tau \in T_\omega(A)$,
there exist pairwise orthogonal projections $p_1, \hdots, p_k \in  A^{\omega,\rm b} \cap S_0' \cap\{\alpha_g^\omega(q) \mid g \in G\}'$ such that for all $\tau \in T_\omega(A)$
\begin{itemize}[itemsep=1ex,topsep=1ex]
\item $\sum_{i=1}^k\tau(p_iq) = \frac{\mu}{k}$,
\item $\tau(a_ip_iq) \leq \delta \tau(p_iq)$ for $i=1, \hdots, k$, 
\item $\displaystyle \max_{g \in F} \sum_{i=1}^k\|q(\alpha^\omega_g(p_i) -p_i)\|_{2,\tau}^2 \leq \frac{2\mu\sqrt{\e}}{k}$, and
\item $\displaystyle \sum_{i=1}^k\|\alpha_g^\omega(p_i)(\alpha_g^\omega(q) - q)\|_{2,\tau}^2 \leq \frac{1}{k}\| \alpha_g^\omega(q) - q\|_{2,T_\omega(A)}^2$ for all $g \in G$.
\end{itemize}
\end{lemma}
\begin{proof}
Let $\e >0$, $F,H \ssubset G$, $\delta >0$ and 
$a_1, \hdots, a_k \in A_+$ be as in the statement of the lemma. In order to prove the claim, it suffices to prove the following local statement:

For every $\mu \in (0,1], \zeta >0$, $T\ssubset A$ and $E \ssubset G$, there exist $S \ssubset A$ and $\xi >0$ such that if $q \in A_+^1$ satisfies
\begin{align*}
\|q-q^2\|_{2,u} &< \xi,\\
\sup_{\tau \in T(A)} |\tau(q) - \mu| &< \xi, \text{ and}\\
\|[q,s]\|_{2,u} &< \xi \text{ for }s \in S,
\end{align*}
then there exist pairwise orthogonal $p_1, \hdots, p_k \in A_+^1$ such that
\begin{align*}
\|p_i-p_i^2\|_{2,u} &< \zeta,\\
\|[p_i,t]\|_{2,u} &< \zeta \text{ for } t \in T \cup \{\alpha_g(q) \mid g \in E\},\\
\sup_{\tau \in T(A)}|\sum_{i=1}^k \tau(p_iq) - \mu/k| &< \zeta,\\
\tau(a_ip_iq) &< \delta\tau(p_iq) + \zeta \text{ for } \tau \in T(A),\\
\sup_{\tau \in T(A)}\max_{g \in F} \sum_{i=1}^k \|q(\alpha_g(p_i) - p_i\|_{2,\tau}^2 &< \frac{2 \mu \sqrt{\e}}{k} +\zeta, \text{ and}\\
\sup_{\tau \in T(A)} \max_{g \in E} \sum_{i=1}^k \|\alpha_g(p_i)^{1/2}(\alpha_g(q) - q)\|_{2,\tau}^2 &< \frac{1}{k}\|\alpha_g(q) - q\|_{2,u}^2 + \zeta.
\end{align*}

We prove this local statement by contradiction. Suppose there exist $\mu \in (0,1]$, $\zeta >0$, $T \ssubset A$ and $E \ssubset G$ for which the statement does not hold.
In other words, this means that for every $\emptyset \neq S \ssubset A$ we can find a $q_{S} \in A_+^1$ such that 
\begin{align*}
\|q_{S}-q_{S}^2\|_{2,u} &< 1/|S|, \\
 \sup_{\tau \in T(A)} |\tau(q_{S}) - \mu| &< 1/|S|, \text{ and}\\
 \|[q_{S},s]\|_{2,u} &< 1/|S| \text{ for } s \in S,
\end{align*}
but there exist no pairwise orthogonal $p_1, \hdots, p_k \in A_+^1$ satisfying 
\begin{align}
\|p_i-p_i^2\|_{2,u} &< \zeta, \label{eq:precursor_under_q_contradiction_1}\\
\|[p_i,t]\|_{2,u} &< \zeta \text{ for } t \in T \cup \{\alpha_g(q_{S}) \mid g \in E\},\\
\sup_{\tau \in T(A)}|\sum_{i=1}^k \tau(p_iq_{S}) - \mu/k| &< \zeta,\\
\tau(a_ip_iq_{S}) &< \delta\tau(p_iq_{S}) + \zeta \text{ for } \tau \in T(A),\\
\sup_{\tau \in T(A)} \max_{g \in F}\sum_{i=1}^k \|q_{S}(\alpha_g(p_i) - p_i\|_{2,\tau}^2 &< \frac{2\mu\sqrt{\e}}{k}+\zeta, \text{ and}\\
\sup_{\tau \in T(A)}\max_{g \in E} \sum_{i=1}^k \|\alpha_g(p_i)^{1/2}(\alpha_g(q_{S}) - q_{S})\|_{2,\tau}^2 &< \frac{1}{k}\|\alpha_g(q_{S}) - q_{S}\|_{2,u}^2 + \zeta.\label{eq:precursor_under_q_contradiction_last}
\end{align}
In this way we get a net $(q_S)_{S}$ indexed by the finite subsets of $A$ equipped with inclusion as the natural partial order. We can take a free ultrafilter $\tilde{\omega}$ on this index set of finite subsets of $A$ as follows. For each $I \ssubset A$ consider the set
\[
P_I = \{J \ssubset A \mid I \subseteq J\}.
\]
As the collection of $P_I$ is closed under finite intersections, they form a filter basis and hence there is a minimal filter on the set of finite subsets of $A$ containing all the sets $P_I$ for $I\ssubset A$.
This filter will be free and can be extended to a free ultrafilter $\tilde{\omega}$ by Zorn's lemma.   

Similarly as in Definition \ref{definition:ultrapowers} and \ref{definition:tracial_ultrapowers}, we can define the norm and bounded tracial ultrapower of $A$ over the ultrafilter $\tilde{\omega}$.
We also get a set of limit traces $T_{\tilde{\omega}}(A)$ over $\tilde{\omega}$.
Then the net $(q_S)_S$ defines a projection $q \in A^{\tilde{\omega},\rm b} \cap A'$ with value $\mu$ on all limit traces on $A^{\tilde{\omega},\rm b}$.
By Lemma \ref{lemma:precursor_2} we can find pairwise orthogonal $p'_1, \hdots,p'_k \in A_+^1$ such that 
\begin{align}
\| p'_i- {p'_i}^2 \|_{2,u} &< \zeta\text{ for }i=1, \hdots, k,\label{eq:precursor_under_q_projection}\\
\| [p'_i,t] \|_{2,u} &< \zeta \text{ for }t \in T, \label{eq:precursor_under_q_central}\\
 \sup_{\tau \in T(A)} \Big| \tau\big( \sum_{i=1}^k p'_i \big) - \frac{1}{k} \Big| &< \frac{1}{4}\zeta, \label{eq:precursor_under_q_sum}\\
\tau(a_ip_i') &< \delta \tau(p'_i) + \zeta \text{ for } i=1, \hdots, k, \, \tau \in T(A), \text{ and}\label{eq:precursor_under_q_delta_inequality}\\
\max_{g \in F} \sum_{i=1}^k \| \alpha_g(p_i')-p_i' \|_{2,\tau}^2 &<\frac{2\sqrt{\e}}{k}\text{ for } \tau \in T(A). \label{eq:precursor_under_q_invariance}
\end{align}
Since $q \in A^{\tilde{\omega},\mathrm{b}} \cap A'$, it follows for each $\tau \in T_{\tilde{\omega}}(A)$ and $g \in G$ that the assignment
\[
A \rightarrow \C: a \mapsto \tau(a\alpha_g^{\tilde{\omega}}(q))/\tau(\alpha_g^{\tilde{\omega}}(q)) =\tau(a\alpha_g^{\tilde{\omega}}(q))/\mu
\]
defines a tracial state on $A$.
In particular, we get that for $\tau \in T_{\tilde{\omega}}(A)$ the following hold:
\begin{align}
\Big| \tau\big( \sum_{i=1}^k p'_i\alpha_g^{\tilde{\omega}}(q) \big) - \frac{\mu}{k} \Big| \overset{\eqref{eq:precursor_under_q_sum}}&{<} \frac{1}{4}\mu\zeta \leq \frac{1}{4}\zeta \text { for } g \in G,   \label{eq:precursor_under_q_main_terms}\\
\tau(a_ip_i'q) \overset{\eqref{eq:precursor_under_q_delta_inequality}}&{<} \delta\tau(p_i'q) + \mu\zeta \leq  \delta\tau(p_i'q) + \zeta, \text{ and} \label{eq:precursor_under_q_delta_inequality_q}\\
\max_{g \in F} \sum_{i=1}^k \| q(\alpha_g(p_i')-p_i') \|_{2,\tau}^2 \overset{\eqref{eq:precursor_under_q_invariance}}&{<} \frac{2 \mu \sqrt{\e}}{k}. \label{eq:precursor_under_q_invariance1}
\end{align}
Next we show that for all $\tau \in T_{\tilde{\omega}}(A)$ and $g \in G$
\begin{equation}\label{eq:precursor_under_q_invariance2}\sum_{i=1}^k \| \alpha_g(p'_i)^{1/2}(\alpha^{\tilde{\omega}}_g(q) - q) \|_{2,\tau}^2 < \frac{1}{k} \|\alpha^{\tilde{\omega}}_g(q) - q \|_{2,u}^2 + \zeta.\end{equation}
 Fix $\tau \in T_{\tilde{\omega}}(A)$ and $g \in G$.
Assume $\tau(q \alpha_g^{\tilde{\omega}}(q)) > 0$. Then the map
\[A \rightarrow \C,\quad  a \mapsto \tau(aq\alpha_g^{\tilde{\omega}}(q))/\tau(q\alpha_g^{\tilde{\omega}}(q))\]
defines a tracial state on $A$.
This means that 
\begin{equation}\label{eq:precursor_under_q_mixed_terms} 
\Big| \tau\big( \sum_{i=1}^k\alpha_g(p'_i)q\alpha_g^{\tilde{\omega}}(q) \big) - \frac{1}{k}\tau(q\alpha_g^{\tilde{\omega}}(q)) \Big| \overset{\eqref{eq:precursor_under_q_sum}}{<} \frac{1}{4}\zeta \tau(q\alpha_g^{\tilde{\omega}}(q)) \leq \frac{1}{4}\zeta.\end{equation}
Note that if $\tau(q \alpha_g^{\tilde{\omega}}(q)) = 0$, then 
\[
\tau\big( \sum_{i=1}^k \alpha_g(p'_i)q\alpha_g^{\tilde{\omega}}(q) \big) = 0 = \frac{1}{k} \tau(q \alpha_g^{\tilde{\omega}}(q))\] and hence, 
\eqref{eq:precursor_under_q_mixed_terms} also holds in this case. 
Now
\begin{align*}
\tau\big( \sum_{i=1}^k \alpha_g(p'_i)(\alpha_g^{\tilde{\omega}}(q) - q)^2 \big) &= \tau\Big( \sum_{i=1}^k \alpha_g(p'_i) \big( \alpha_g^{\tilde{\omega}}(q) + q - \alpha_g^{\tilde{\omega}}(q) q - q \alpha_g^{\tilde{\omega}}(q) \big) \Big)\\
 \overset{\eqref{eq:precursor_under_q_main_terms},\eqref{eq:precursor_under_q_mixed_terms}}&{<} \frac{1}{k}\tau(\alpha_g^{\tilde{\omega}}(q) + q - \alpha_g^{\tilde{\omega}}(q) q - q \alpha_g^{\tilde{\omega}}(q)) + \zeta\\
&= \frac{1}{k}\tau((\alpha_g^{\tilde{\omega}}(q) - q)^2) +\zeta.
\end{align*}
This proves \eqref{eq:precursor_under_q_invariance2}. 
If we combine this with equations \eqref{eq:precursor_under_q_projection}, \eqref{eq:precursor_under_q_central} and \eqref{eq:precursor_under_q_main_terms}--\eqref{eq:precursor_under_q_invariance1}, we can conclude that for some $q_S$ in the net, equations \eqref{eq:precursor_under_q_contradiction_1} -- \eqref{eq:precursor_under_q_contradiction_last} must hold. This gives the desired contradiction.
\end{proof}

\begin{proof}[Proof of Lemma \ref{lemma:equivariant_CPOU}]
Given $\e >0$ and $t \in (0,1)$, choose $\eta >0$ sufficiently small such that 
\begin{equation}\label{eq:choice_eta} 4\left\lceil \frac{t}{1-t} \right\rceil \sqrt{\eta} < \e.\end{equation}
We show that such a constant $\eta$ satisfies the required properties. 
Let $\alpha:G \acts A$, finite sets $F,H \ssubset G$, and $S \subset A^{\omega,\rm b}$ be given as in the statement of the lemma. By Remark \ref{remark:reduction_constant_elements} it suffices to consider $\delta >0$ and $a_1, \hdots, a_k \in A_+$ such that
\[\frac{\delta}{|H|} > \sup_{\tau \in T(A)}\min_{i = 1, \hdots, k} \tau(a_i).\]

We construct the pairwise orthogonal projections $p_1, \hdots, p_k$ in $N:= k \lceil \frac{t}{1-t}\rceil$ steps. 
Define a sequence $(s_n)$ in $[0,1)$ inductively by setting $s_0 = 0$ and setting $s_{i+1} = s_i + \frac{1}{k}(1-s_i)$.
Note that when $s < t$, then $s+\frac{1-s}{k} > s+\frac{1-t}{k}$.
If we assumed for a moment that $s_N<t$, then this sequence is less than $t$ for all of the first $N$ steps, leading to $s_N > N\frac{1-t}{k}\geq t$, which is a contradiction, hence we must have $s_N >t$.
Next, we construct separable subsets $S_0, \hdots, S_N$ as follows:
Set $S_N := S$.
Given $i\in\{1,\dots,N\}$ such that $S_i$ is defined, let $S_{i-1}$ be the union of $S_i$ and the set determined by Lemma \ref{lemma:precursor_final} with $1-s_{i-1}$ in place of $\mu$ and $S_{i}$ in place of $S_0$. 

In the initial step, we set $p_1^{(0)}= \hdots = p_k^{(0)}=0$.
Now suppose that for some $n \in \N$ we have pairwise orthogonal projections $p_1^{(n)}, \hdots, p_k^{(n)} \in A^{\omega,\rm b} \cap S_n'$ such that for all $\tau \in T_\omega(A)$:
\begingroup
\allowdisplaybreaks
\begin{align}
\tau(p_1^{(n)} + \hdots + p_k^{(n)}) &= s_n,\label{eq:sum_n}\\
\tau(a_ip_i^{(n)}) &\leq \delta \tau(p_i^{(n)}) \text{ for } i=1, \hdots, k, \label{eq:delta_inequality_n}\\
\max_{g \in F}\sum_{i=1}^k\|\alpha^\omega_g(p_i^{(n)})-p_i^{(n)}\|_{2,\tau}^2 &\leq \frac{4n\sqrt{\eta}}{k}, \text{ and} \label{eq:sum_invariances_n} \\
\max_{g \in F} \big\| \sum_{i=1}^k\alpha^\omega_g(p_i^{(n)})-\sum_{i=1}^kp_i^{(n)} \big\|_{2,T_\omega(A)}^2 &\leq 2 s_n \sqrt{\eta}.\label{eq:invariance_sum_n}
\end{align}
\endgroup
Note that the $p_i^{(0)}$ trivially satisfy equations \eqref{eq:sum_n}--\eqref{eq:invariance_sum_n}.
We show that we can construct pairwise orthogonal projections $p_1^{(n+1)}, \hdots, p_k^{(n+1)} \in A^{\omega,\rm b} \cap S_{n+1}'$ such that for all $\tau \in T_\omega(A)$:
\begin{align}
\tau(p_1^{(n+1)} + \hdots + p_k^{(n+1)}) &= s_{n+1}, \label{eq:sum_n+1}\\
\tau(a_ip_i^{(n+1)}) &\leq \delta \tau(p_i^{(n+1)}) \text{ for } i=1, \hdots, k, \label{eq:delta_inequality_n+1}\\
\max_{g \in F}\sum_{i=1}^k\|\alpha^\omega_g(p_i^{(n+1)})-p_i^{(n+1)}\|_{2,\tau}^2 &\leq \frac{4(n+1)\sqrt{\eta}}{k}, \text{ and } \label{eq:sum_invariances_n+1}\\
\max_{g \in F} \big\| \sum_{i=1}^k\alpha^\omega_g(p_i^{(n+1)})-\sum_{i=1}^k p_i^{(n+1)} \big\|_{2,T_\omega(A)}^2 &\leq 2s_{n+1}\sqrt{\eta}.\label{eq:invariance_sum_n+1}
\end{align}
Define $q:= 1_{A^{\omega,\rm b}}- \sum_{i=1}^k p_i^{(n)}$.
Note that $q$ is a projection in $A^{\omega,\rm b} \cap S_n'$ with $\tau(q) = 1-s_n$ for all $\tau \in T_\omega(A)$.
By Lemma \ref{lemma:precursor_final} and our choice of $S_n$ we can find pairwise orthogonal projections $r_1, \hdots, r_k \in A^{\omega,\rm b} \cap S_{n+1}' \cap \{\alpha_g^\omega(q)\mid g \in G\}'$ such that for all $\tau \in T_\omega(A)$:
\begingroup
\allowdisplaybreaks
\begin{align}
\tau(\sum_{i=1}^k r_iq) &= \frac{1-s_n}{k}, \label{eq:sum_leftover}\\
\tau(a_ir_iq) &\leq \delta \tau(r_iq) \text{ for } i=1, \hdots, k, \label{eq:delta_inequality_leftover}\\
\max_{g \in F}\sum_{i=1}^k\|q(\alpha^\omega_g(r_i)-r_i)\|_{2,\tau}^2 &\leq \frac{2(1-s_n)\sqrt{\eta}}{k}, \text{ and} \label{eq:sum_invariances_leftover}\\
\sum_{i=1}^k\|\alpha_g^\omega(r_i)(\alpha_g^\omega(q) - q)\|_{2,\tau}^2 &\leq \frac{1}{k}\| \alpha_g^\omega(q) - q\|_{2,T_\omega(A)}^2 \text{ for } g \in G.\label{eq:sum_invariances_leftover_2}
 \end{align}
\endgroup
Define
\[
p^{(n+1)}_i := p^{(n)}_i + qr_i.
\]
By construction (recall that $S_{n+1}\subset S_n$), the elements $p_1^{(n+1)}, \hdots, p_k^{(n+1)}$ are pairwise orthogonal projections in $ A^{\omega,\rm b} \cap S_{n+1}'$.
We show that they satisfy equations \eqref{eq:sum_n+1}--\eqref{eq:invariance_sum_n+1}.
Fix $\tau \in T_\omega (A)$.
Note first of all that
\[
\tau\big( \sum_{i=1}^k p_i^{(n+1)} \big) = \tau\big( \sum_{i=1}^k p_i^{(n)} \big) + \tau\big( \sum_{i=1}^k qr_i \big) \overset{\eqref{eq:sum_n}, \eqref{eq:sum_leftover}}{=} s_n + \frac{1}{k}(1-s_n)=s_{n+1}.
\]
Moreover, for $i=1, \hdots, k$ we have
\begin{align*}
\tau(a_ip_i^{(n+1)}) & = \tau(a_ip_i^{(n)}) + \tau(a_iqr_i)\\
\overset{\eqref{eq:delta_inequality_n},\eqref{eq:delta_inequality_leftover}}&{\leq}\delta \tau(p_i^{(n)}) + \delta \tau(qr_i) = \delta \tau(p_i^{(n+1)}).
\end{align*}
This shows already that the projections satisfy \eqref{eq:sum_n+1} and \eqref{eq:delta_inequality_n+1}. Next we prove that they satisfy \eqref{eq:sum_invariances_n+1}.  Note that $p_i^{(n)}$ is orthogonal to $q$ for $i=1, \hdots, k$. Hence, we get for each $g \in G$ and $\tau \in T_\omega(A)$ that
\begingroup
\allowdisplaybreaks
\begin{align}
&\sum_{i=1}^k \| \alpha^\omega_g(p_i^{(n+1)})-p_i^{(n+1)} \|_{2,\tau}^2 \nonumber\\
 &= \sum_{i=1}^k \tau\big( (\alpha^\omega_g(p_i^{(n+1)})-p_i^{(n+1)})^2 \big) \nonumber\\
&= \sum_{i=1}^k \Big( \tau\big( (\alpha^\omega_g(p_i^{(n)})-p_i^{(n)})^2 \big) + \tau\big( (\alpha^\omega_g(q r_i)- qr_i)^2 \big) -2 \tau\big( \alpha_g^\omega(p_i^{(n)})qr_i \big) - 2 \tau\big( p_i^{(n)}\alpha_g^\omega(qr_i) \big) \Big) \nonumber\\
&\leq  \sum_{i=1}^k \| \alpha^\omega_g(p_i^{(n)})-p_i^{(n)} \|_{2,\tau}^2 + \sum_{i=1}^k \|\alpha^\omega_g(q r_i)- qr_i \|_{2,\tau}^2.\label{eq:calculation_sum_invariances_n+1}
\end{align}
\endgroup
For $i=1, \hdots, k$ we have
\begin{align*}
\|\alpha^\omega_g(q r_i)- qr_i\|_{2,\tau}^2 & \leq \big( \|(\alpha^\omega_g(q)- q)\alpha_g^\omega(r_i)\|_{2,\tau}+ \|q(\alpha^\omega_g(r_i)- r_i)\|_{2,\tau} \big)^2\\
& \leq  2\big( \|(\alpha^\omega_g(q)- q)\alpha_g^\omega(r_i)\|_{2,\tau}^2 + \|q(\alpha^\omega_g(r_i)- r_i)\|_{2,\tau}^2 \big).
\end{align*}
Combining this with \eqref{eq:calculation_sum_invariances_n+1} we find that for $g \in F$
\begingroup
\allowdisplaybreaks
\begin{align*}
&\sum_{i=1}^k\|\alpha^\omega_g(p_i^{(n+1)})-p_i^{(n+1)}\|_{2,\tau}^2\\
 &\leq \sum_{i=1}^k \|\alpha^\omega_g(p_i^{(n)})-p_i^{(n)}\|_{2,\tau}^2 + 2\sum_{i=1}^k \left( \|(\alpha^\omega_g(q)- q)\alpha_g^\omega(r_i)\|_{2,\tau}^2 + \|q(\alpha^\omega_g(r_i)- r_i)\|_{2,\tau}^2\right)\\
 \overset{\eqref{eq:sum_invariances_n},\eqref{eq:sum_invariances_leftover_2}}&{\leq} \frac{4n\sqrt{\eta}}{k} + \frac{2}{k} \Big\| \sum_{i=1}^k\alpha^\omega_g(p_i^{(n)})- \sum_{i=1}^k p_i^{(n)} \Big\|^2_{2,T_\omega(A)}  + 2\sum_{i=1}^k\|q(\alpha^\omega_g(r_i)- r_i)\|_{2,\tau}^2 \\
 \overset{\eqref{eq:invariance_sum_n},\eqref{eq:sum_invariances_leftover}}&{\leq} \frac{4n\sqrt{\eta}}{k} +  \frac{4s_n\sqrt{\eta}}{k} + \frac{4(1-s_n)\sqrt{\eta}}{k} \\ 
 &= \frac{4(n+1)\sqrt{\eta}}{k}.
\end{align*}
\endgroup
Lastly, we show that the elements $p_i^{(n+1)}$ satisfy \eqref{eq:invariance_sum_n+1}.
We get for each $g \in G$ that
\begingroup
\allowdisplaybreaks
\begin{align*} &\sum_{i=1}^k ( \alpha^\omega_g(p_i^{(n+1)}) - p_i^{(n+1)} ) \\ &= \sum_{i=1}^k (\alpha^\omega_g(p_i^{(n)}) - p_i^{(n)}) + \alpha^\omega_g\Big( q\Big(\sum_{j=1}^k r_j\Big) \Big) - q\Big(\sum_{j=1}^k r_j\Big) \\
&= q-\alpha_g^\omega(q)+ \alpha^\omega_g\Big( q\Big(\sum_{j=1}^k r_j\Big) \Big) - q\Big(\sum_{j=1}^k r_j\Big) \\
&= q\Big(1-\sum_{j=1}^k r_j\Big) - \alpha^\omega_g\Big( q\Big(1-\sum_{j=1}^k r_j\Big) \Big) \\
&= q\Big(\sum_{j=1}^k ( \alpha^\omega_g(r_j) - r_j) \Big) + q\Big(1-\sum_{j=1}^k \alpha^\omega_g(r_j) \Big) - \alpha^\omega_g\Big( q\Big(1-\sum_{j=1}^k r_j\Big) \Big) \\
&= q\Big(\sum_{j=1}^k ( \alpha^\omega_g(r_j) - r_j) \Big) + (q-\alpha^\omega_g(q))\Big(1-\sum_{j=1}^k \alpha^\omega_g(r_j) \Big).
\end{align*}
\endgroup
Keeping in mind that $q$ is a projection and the elements $r_j$ are pairwise orthogonal projections, we have for all $g \in F$ and $\tau \in T_\omega(A)$ that 
\begingroup
\allowdisplaybreaks
\begin{align*}
 &\Big\| \sum_{i=1}^k\alpha^\omega_g(p_i^{(n+1)})-\sum_{i=1}^k p_i^{(n+1)} \Big\|_{2,\tau}^2  \\
&= \tau\left( \Big( \sum_{i=1}^k \alpha_g^\omega(p_i^{(n+1)}) - \sum_{i=1}^k p_i^{(n+1)} \Big)^2 \right)\\
&= \tau\left( \Big( \alpha_g^\omega(p_i^{(n)}+qr_i) - \sum_{i=1}^k p_i^{(n)} + qr_i \Big)^2 \right) \\
&= \tau\left( \Big( \sum_{i=1}^k q \big( \alpha_g^{\omega}(r_i)-r_i \big) + (q-\alpha_g^\omega(q))\big(1-\alpha_g^\omega(r_i)\big) \Big)^2 \right) \\
&=\tau\Bigg( q \Big( \sum_{i,j=1}^k \alpha^\omega_g(r_i r_j) + r_i r_j - r_i \alpha^\omega_g(r_j) - \alpha_g^\omega(r_i)r_j \Big) + (q-\alpha^\omega_g(q))^2\Big(1-\sum_{j=1}^k \alpha_g^\omega(r_j) \Big) \Bigg)\\
&\qquad - 2\tau\Bigg(\Big(\sum_{j=1}^k r_j \Big)q(1-\alpha_g^\omega(q))\Big(1-\sum_{j=1}^k \alpha^\omega_g(r_j) \Big)\Bigg)\\
&\leq  \tau\Bigg( q \Big( \sum_{i,j=1}^k \alpha^\omega_g(r_i r_j) + r_i r_j - r_i \alpha^\omega_g(r_j) - \alpha_g^\omega(r_i)r_j \Big) + (q-\alpha^\omega_g(q))^2 \Bigg)\\
&\leq  \tau\Bigg( q \Big( \sum_{i=1}^k \alpha^\omega_g(r_i) + r_i - r_i \alpha^\omega_g(r_i) - \alpha_g^\omega(r_i)r_i \Big) + (q-\alpha^\omega_g(q))^2 \Bigg)\\
&= \tau\Bigg( q \Big( \sum_{i=1}^k (\alpha^\omega_g(r_i) - r_i)^2 \Big)  + (q-\alpha^\omega_g(q))^2 \Bigg)\\
&= \sum_{i=1}^k \|q(r_i - \alpha_g^\omega(r_i)) \|_{2,\tau}^2 + \| q-\alpha^\omega_g(q) \|_{2,\tau}^2 \\
\overset{\eqref{eq:invariance_sum_n},\eqref{eq:sum_invariances_leftover}}&{\leq} 2(s_n+\frac{1}{k}(1-s_n))\sqrt{\eta} =2 s_{n+1}\sqrt{\eta}.
\end{align*}
\endgroup
In the first inequality in the above computation, we used the fact that $q$ commutes with all the elements $r_i$, thus the term $\Big(\sum_{j=1}^kr_j \Big)q(1-\alpha_g^\omega(q))\Big(1-\sum_{j=1}^k \alpha^\omega_g(r_j)\Big)$ is a product of two positive elements, whose trace value must be nonnegative.
In the second inequality we used the pairwise orthogonality of the projections $r_i$, so the mixed terms in the double sum appearing above contribute the trace value of $-(r_i\alpha_g^{\omega}(r_j)+\alpha_g^\omega(r_i)r_j)$, which likewise is nonpositive.

As explained before, one has $s_N > t$.
So if we start the inductive procedure with the projections $p_1^{(0)}= \hdots= p_k^{(0)}=0$, then after $N$ steps we obtain projections $p_1^{(N)}, \hdots, p_k^{(N)}$ satisfying
\[\sum_{i=1}^k \tau(p_i^{(N)}) > t \quad \text{for } \tau \in T_\omega(A).\]
Moreover, at that point we have
\[\tau(a_ip_i^{(N)}) \overset{\eqref{eq:delta_inequality_n+1}}{\leq} \delta \tau(p_i^{(N)}) \text{ for } i=1, \hdots, k, \, \tau \in T_\omega(A), \text{ and}\]
\[\max_{g \in F}\sum_{i=1}^k\|\alpha^\omega_g(p_i^{(N)})-p_i^{(N)}\|_{2,\tau}^2
\overset{\eqref{eq:sum_invariances_n+1}}{\leq} \frac{4N\sqrt{\eta}}{k} \overset{\eqref{eq:choice_eta}}{<} \e.\]
We conclude that $p_1^{(N)}, \hdots, p_k^{(N)}$ satisfy all the required properties. 
\end{proof}


\section{Dynamical tracial local-to-global principle}
In this section we prove our main technical result, namely that equivariant property Gamma implies a \emph{tracial local-to-global principle} for actions of amenable groups.
Roughly, this means that whenever a $*$-polynomial identity has (local) approximate solutions one tracial presentation at a time, then it has (global) approximate solutions in the uniform tracial 2-norm.
We begin by making precise what we mean by these polynomial identities:

\begin{definition}[{cf.\ \cite[Definition 4.4]{GardellaHirshbergVaccaro}}]
Let $G$ be a discrete group and let $X$ be a countable set of non-commutative variables. A \emph{non-commutative $G$-$*$-polynomial} in the variables $X$ is a non-commutative $*$-polynomial in the variables $\{g\cdot x \mid g \in G, x \in X\}$. 

Let $A$ be a \cstar-algebra with action $\alpha:G \acts A$. 
Suppose that $h(x_1, \hdots, x_r)$ is a $G$-$*$-polynomial in $r$ non-commuting variables. Given a tuple $(a_1, \hdots, a_r) \in A^r$, the evaluation $h(a_1, \hdots, a_r)$ is computed by interpreting $g \cdot x_i$ as $\alpha_g(a_i)$ for $g \in G$ and $i=1, \hdots, r$. 
\end{definition}

The main theorem that we prove in this section is the following:
\begin{theorem}
\label{theorem:local-to-global}
Let $A$ be a $\sigma$-unital \cstar-algebra with $T(A)$ non-empty and compact, and with weak \textup{CPoU}.
Let $\alpha:G \acts A$ be an action by an amenable countable discrete group and assume it has local equivariant property Gamma w.r.t.\ bounded traces.
For each $m \in \N$, let
\[
h_m(x_1, \hdots, x_{r_m}, z_1, \hdots, z_{s_m})
\]
be a $G$-$*$-polynomial in $r_m + s_m$ non-commuting variables.
Let $(a_i)_{i \in \N}$ be a sequence in $A^{\omega,\rm b}$.
Suppose for every $\e >0$, $\ell \in \N$ and $\tau \in \overline{T_\omega(A)}^{w^*}$, there exist contractions $(y_i^\tau)_{i \in \N}$ in  $A^{\omega,\rm b}$ such that 
\[
\| h_m(a_1, \hdots, a_{r_m}, y_1^\tau, \hdots, y_{s_m}^\tau) \|_{2,\tau} < \e \quad \text{ for } m =1, \hdots, \ell.
\]
Then there exist contractions $(y_i)_{i \in \N}$ in $A^{\omega, \rm b}$  such that 
\begin{equation}\label{eq:uniform_solutions}
h_m(a_1, \hdots, a_{r_m}, y_1, \hdots, y_{s_m})=0 \quad \text{for all } m \in \N.
\end{equation}
\end{theorem}

In order to simplify the proof of this result we consider the following technical lemma, which reduces the complexity of the involved polynomials:

\begin{lemma}\label{lemma:simplify_polynomials}
Let $G$ be a countable discrete group.
Consider sets of variables $X = \{x_i\mid i \in \N\}$ and $Z =\{z_i \mid i \in \N\}$.
Assume that 
\[
\mathcal{P} = \{h_m(x_1, \hdots, x_{r_m}, z_1, \hdots, z_{s_m}) \mid m \in \N\}
\]
is a countable set of non-commutative $G$-$*$-polynomials in the variables $X \cup Z$.
Then there exists another set of variables $Z'=\{z'_i\mid i \in \N\}$ and another countable set of non-commutative $G$-$*$-polynomials
\[
\mathcal{P}' = \{h'_m(x_1, \hdots, x_{r'_m},z'_1, \hdots z'_{s'_m}) \mid m \in \N\}
\]
in the variables $X \cup Z'$ such that every $G$-$*$-polynomial $h'_m$ satisfies one of the following properties:
\begin{enumerate}[itemsep=1ex,topsep=1ex, leftmargin=*,label=\textup{(\arabic*)}]
\item $h'_m(x_1, \hdots, x_{r'_m},0, \hdots, 0) = 0 $ (i.e., no terms in the polynomial with variables only in $X$) and $h'_m(1, \hdots, 1, z'_1, \hdots, z'_{s'_m})$ is an ordinary $*$-polynomial in the variables $Z'$,\label{enum:no_actions_no_constant_terms}
\item $h'_m(x_1, \hdots, x_{r'_m}, z'_1, \hdots, z'_{s'_m})= \|h'_m(x_1, \hdots, x_{r'_m}, 0, \hdots, 0)\|z'_{s'_m} - h'_m(x_1, \hdots, x_{r'_m}, 0, \hdots, 0)$,\label{enum:constant_term}
\item $h'_m(x_1, \hdots, x_{r'_m},z'_1, \hdots, z'_{s'_m}) = g\cdot z'_i - z'_j $ for some $1 \leq i,j \leq s'_m$ and some $g \in G$.\label{enum:action}
\end{enumerate}
Moreover we have that for every action $\beta:G \acts B$ on any \cstar-algebra, any sequence $(b_i)_{i \in \N}$ in $B$ and subset $T \subset T(B)$ the following two statements hold:
\begin{enumerate}[label=\textup{(\alph*)}]
\item\label{equivalence_solutions_exact} There exist contractions $(y_i)_{i \in \N}$ in $B$ such that
\[\|h_m(b_1, \hdots, b_{r_m}, y_1, \hdots, y_{s_m})\|_{2,T} = 0 \quad \text{ for  all } m \in \N\]
if and only if there exist contractions $(y'_i)_{i \in \N}$ in $B$ such that
\[\|h'_m(b_1, \hdots, b_{r'_m}, y'_1, \hdots, y'_{s'_m})\|_{2,T} = 0 \quad \text{ for all } m \in \N.\]
\item\label{equivalence_solutions_approximate} For each $\e >0$ and each $\ell \in \N$ there exist contractions  $(y_i)_{i \in \N}$ in $B$ such that
\[\|h_m(b_1, \hdots, b_{r_m}, y_1, \hdots, y_{s_m})\|_{2,T} < \e \quad \text{ for } m = 1 , \hdots, \ell \]
if and only if for each $\e' >0$ and each $\ell' \in \N$ there exist contractions $(y'_i)_{i \in \N}$ in $B$ such that
\[\|h'_m(b_1, \hdots, b_{r'_m}, y'_1, \hdots, y'_{s'_m})\|_{2,T} < \e' \quad \text{ for } m=1, \hdots, \ell'.\]
\end{enumerate}
\end{lemma}
\begin{proof}
Define $Z' =\{z_i\mid i \in \N\} \cup \{z_{i,g} \mid i \in \N, g \in G\setminus\{e\}\} \cup\{w_m \mid m \in \N\}$. For each $m \in \N$ we can take the $G$-$*$-polynomial $h_m$ in the variables $X \cup Z$ and define $h''_m$ in the variables $X \cup Z'$ by replacing every instance of a variable $g \cdot z_i$ for some $i \in \N$, $g \in G\setminus\{e\}$ by $z_{i,g}$, e.g.\ the polynomial $g \cdot z_1 - z_2$ would be transformed into $z_{1,g}- z_{2}$. Next, we define a new $G$-$*$-polynomial $h'''_m$ for each $m \in \N$ by setting
\begin{equation}\label{eq:changed_polynomial} h'''_m(X \cup Z') = h''_m (X \cup Z') + \|h_m(x_1, \hdots, x_{r_m}, 0, \hdots, 0)\|w_m - h_m(x_1, \hdots, x_{r_m}, 0, \hdots, 0).\end{equation}
Set $\mathcal{P}_1' :=\{h'''_m(X \cup Z') \mid m \in \N\}$. All the polynomials in this set are of the type \ref{enum:no_actions_no_constant_terms} mentioned in the statement of this lemma. 
Next, define the sets of $G$-$*$-polynomials
\[\mathcal{P}'_2 :=\{\|h_m(x_1, \hdots, x_{r_m}, 0, \hdots, 0)\|w_m - h_m(x_1, \hdots, x_{r_m}, 0, \hdots, 0)\mid m \in \N\},\]
and
\[\mathcal{P}'_3 :=\{g \cdot z_i - z_{i,g} \mid i \in \N, g \in G \setminus \{e\}\}.\]
These sets consist of polynomials of type \ref{enum:constant_term} and \ref{enum:action}, respectively. 

Consider $\mathcal{P}' = \mathcal{P}'_1 \cup \mathcal{P}'_2 \cup \mathcal{P}'_3$.
The $G$-$*$-polynomials in this set are all of the right form and we claim that this does the job.
It suffices to show part \ref{equivalence_solutions_exact}.
This is because for a given \cstar-algebra $B$, action $\beta:G \acts B$, sequence $(b_i)_{i\in \N} \in B$ and subset $T \subset T(B)$, statement \ref{equivalence_solutions_approximate} immediately follows from statement \ref{equivalence_solutions_exact} when applied to the \cstar-algebra $B_\omega$ with action $\beta_\omega$, sequence $(b_i)_{i \in \N}$ in $B \subset B_\omega$, and the set of limit traces on $B_\omega$ arising from sequences of traces in $T$. 

To show part \ref{equivalence_solutions_exact}, fix a \cstar-algebras $B$, an action $\beta: G \acts B$, a sequence $(b_i)_{i \in \N}$ and subset $T \subset T(B)$.
Let $(y_i)_{i \in \N}$ be any sequence of contractions in $B$.
For notational brevity, we denote these sequences by $\bar{b}=(b_i)_i$ and $\bar{y}=(y_i)_i$.
Furthermore we shall also write (exclusively in this proof) for two elements $x,y\in B$ the expression ``$x=_T y$'' as shorthand for $\|x-y\|_{2,T}=0$.

We set $y_{i,g}=\beta_g(y_i)$ for $i\in\N$ and $g\in G\setminus\{e\}$ and
\[
w_m=\begin{cases} 0 & ,\quad h_m(b_1, \hdots, b_{r_m}, 0, \hdots, 0)=0 \\
\frac{h_m(b_1, \hdots, b_{r_m}, 0, \hdots, 0)}{\|h_m(b_1, \hdots, b_{r_m}, 0, \hdots, 0)\|} &,\quad h_m(b_1, \hdots, b_{r_m}, 0, \hdots, 0)\neq 0.\end{cases}
\]
Then the tuple $\bar{z}:=(y_i)_{i\in\N}\times (y_{i,g})_{i\in\N, g\in G\setminus\{e\}}\times (w_m)_{m\in\N}$ represents a choice for the free variables of $Z'$ inside $B$.
By definition we have $p(\bar{b},\bar{z})=0$ for all $p\in \mathcal{P}_2'\cup\mathcal{P}_3'$.
By definition of the polynomials $h_m'''$, we have 
\[
h_m'''(b_1, \hdots, b_{r_m},\bar{z})=h_m''(b_1, \hdots, b_{r_m},\bar{z})+q(\bar{b},\bar{z})=h_m''(b_1, \hdots, b_{r_m},\bar{z})
\]
for some $*$-polynomial $q\in\mathcal{P}_2'$.
Due to the vanishing of all the $*$-polynomials of $\mathcal{P}_3'$ in $\bar{z}$ and given how the polynomial $h_m''$ arises from the polynomial $h_m$ via substituion of variables, we may finally observe
\[
h_m(b_1, \hdots, b_{r_m},y_1,\dots,y_{s_m})=h_m''(b_1, \hdots, b_{r_m},\bar{z}).
\]
This shows immediately that if the sequence $(y_i)_i$ satisfies
\[
\|h_m(b_1, \hdots, b_{r_m}, y_1, \hdots, y_{s_m})\|_{2,T} = 0 \quad \text{ for  all } m \in \N,
\]
then we also have $\| p(\bar{b},\bar{z})\|_{2,T}=0$ for all $p\in\mathcal{P}'$.
In particular, we get the ``only if'' part in \ref{equivalence_solutions_exact}.

Conversely, suppose that $\bar{z}:=(y_i)_{i\in\N}\times (y_{i,g})_{i\in\N, g\in G\setminus\{e\}}\times (w_m)_{m\in\N}$ is an arbitrary tuple with values in the unit ball of $B$ representing a choice for the free variables in $Z'$ such that $p(\bar{b},\bar{z})=0$ for all $p\in\mathcal{P}'$.
By doing the above computations in reverse, we can see that $p(\bar{z})=_T 0$ for $p\in\mathcal{P}'_3$ forces $y_{i,g}=_T \beta_g(y_i)$ for all $g\in G\setminus\{e\}$.
Moreover, the vanishing $p(\bar{b},\bar{z})=_T 0$ for $p\in\mathcal{P}'_2$ forces the equation $w_m=\frac{h_m(b_1, \hdots, b_{r_m}, 0, \hdots, 0)}{\|h_m(b_1, \hdots, b_{r_m}, 0, \hdots, 0)\|}$ when $h_m(b_1, \hdots, b_{r_m}, 0, \hdots, 0)\neq 0$.
Similar to how we argued above, this implies for all $m\geq 1$ that
\[
h_m'''(b_1, \hdots, b_{r_m},\bar{z}) = h_m''(b_1, \hdots, b_{r_m},\bar{z})+q(\bar{b},\bar{z}) =_T h_m''(b_1, \hdots, b_{r_m},\bar{z})
\]
for some $*$-polynomial $q\in\mathcal{P}_2'$.
Given how the polynomial $h_m''$ arises from the polynomial $h_m$ via substituion of variables, we may finally observe
\[
h_m(b_1, \hdots, b_{r_m},y_1,\dots,y_{s_m})=_T h_m''(b_1, \hdots, b_{r_m},\bar{z}),\quad m\geq 1.
\]
This shows the ``if'' part of \ref{equivalence_solutions_exact} and finishes the proof.
\end{proof}

\begin{proof}[Proof of Theorem \ref{theorem:local-to-global}]
Equation \eqref{eq:uniform_solutions} is equivalent to 
\[\|h_m(a_1, \hdots, a_{r_m}, y_1, \hdots, y_{s_m})\|_{2,T_\omega(A)} = 0 \quad \text{for all } m \in \N.\]
By the previous lemma, we may assume that the $*$-polynomials $h_m$ are all of one of the following three types:
\begin{enumerate}[itemsep=1ex,topsep=1ex]
\item $h_m(a_1, \hdots, a_{r_m}, 0, \hdots, 0) = 0$ and $h_m(1_{A^{\omega,\mathrm{b}}}, \hdots, 1_{A^{\omega,\mathrm{b}}},z_1, \hdots, z_{s_m})$ is an ordinary $*$-polynomial.
\item $h_m$ is of the form $\|h_m(x_1, \hdots, x_{r_m}, 0 , \hdots, 0)\|z_i - h_m(x_1, \hdots, x_{r_m}, 0 , \hdots, 0)$ for some $i \in \N$.
Equivalently, since this doesn't change the solutions, we may assume that $h_m(a_1, \hdots, a_{r_m}z_1, \hdots, z_{s_m})$ is of the form $z_i - a$ for some $i \in \N$ and $a \in \mathrm{C}^*\big( \{ \alpha_g^\omega(a_i) \mid i \in \N,\ g\in G\} \big)$ with $\|a\|=1$. 
\item $h_m$ is of the form $g\cdot z_i - z_i'$ for some $i, i'\in \N$ and $g \in G$.
\end{enumerate}

By Kirchberg's $\e$-test, it suffices to find for each $\e >0$ and $\ell \in \N$ contractions $(y_i)_{i \in \N}$ in $A^{\omega,\rm b}$ such that
\[
\|h_m(a_1, \hdots, a_{r_m}, y_1, \hdots, y_{s_m})\|_{2,\tau} \leq \e \quad \text{for } m =1, \hdots, \ell \text{ and } \tau \in T_\omega(A).
\]
Choose $\e>0$ and $\ell \in \N$ arbitrarily. Denote by $F \ssubset G$ the set of $g \in G$ appearing in $h_m$ for some $m=1, \hdots, \ell$. 
Set 
\begin{equation}\label{eq:choose_delta}\delta = \frac{\e^2}{18}.\end{equation}
Choose $t \in (0,1)$ such that 
\begin{equation}\label{eq:choose_t} 1-t < \e^2/2.\end{equation}
Let $\eta>0$ be the universal constant from Lemma \ref{lemma:equivariant_CPOU} corresponding to the tuple $(\delta^2,t)$.
  Since $G$ is amenable we can find $H \ssubset G$ such that $|gH \Delta H| < \eta|H|$ for each $g \in F$. 
By assumption, for each $\tau \in \overline{T_\omega(A)}^{w^*}$ we can find contractions $(y_i^\tau)_{i \in \N} \in A^{\omega,\rm b}$ such that
\[
\|h_m(a_1, \hdots, a_{r_m}, y_1^\tau, \hdots y_{s_m}^\tau)\|^2_{2,\tau} <  \frac{\e^2}{2\ell|H|}\quad \text{for } m=1, \hdots, \ell.
\]
Define
\begin{equation}\label{eq:definition_b}
b^\tau := \sum_{m=1}^\ell |h_m(a_1, \hdots, a_{r_m}, y_1^\tau, \hdots y_{s_m}^\tau)|^2 \in A^{\omega,\rm b}.
\end{equation}
Then we get
\[
\tau(b^\tau) = \sum_{m=1}^\ell \|h_m(a_1, \hdots, a_{r_m}, y_1^\tau, \hdots y_{s_m}^\tau)\|^2_{2,\tau}  < \frac{\e^2}{2|H|}.
\]
By continuity and compactness of $\overline{T_\omega(A)}^{w^*}$, we can find finitely many tracial states $\tau_1, \hdots, \tau_k \in \overline{T_\omega(A)}^{w^*}$ such that
\[ \frac{\e^2}{2|H|} > \sup_{\tau \in T_\omega(A)} \min_{i=1, \hdots, k} \tau(b^{\tau_i}). \]
By Lemma \ref{lemma:equivariant_CPOU}, it follows that we can find pairwise orthogonal projections
\[
p_1, \hdots, p_k \in A^{\omega,\rm b} \cap \Big( \bigcup_{g\in G} \alpha^\omega_g\big( \{y_i^{\tau_1}, \hdots, y_i^{\tau_n}, a_i \mid i \in \N\} \big) \Big)'
\]
such that for $\tau \in T_\omega(A)$
\begin{align}
\tau(\sum_{j=1}^k p_j) &> t, \label{eq:sum}\\
\tau(b^{\tau_j}p_j) &\leq \frac{\e^2}{2}\tau(p_j) \text{ for } j=1, \hdots, k, \text{ and } \label{eq:delta_inequality}\\
\max_{g \in F} \sum_{j=1}^k \|\alpha_g(p_j) - p_j\|^2_{2,\tau} &< \delta^2.\label{eq:invariance}
\end{align}
Define $y_i = \sum_{j=1}^k p_j y_i^{\tau_j}$ for $i \in\N$. 
We show that for $ \tau \in T_\omega(A)$ and $m=1, \hdots, \ell$
\begin{equation}\label{eq:bringing_projections_out}
\Big| \tau\big( |h_m(a_1, \hdots, a_{r_m}, y_1, \hdots, y_{s_m})|^2 \big) - \tau\Big( \sum_{j=1}^k p_j|h_m(a_1,\hdots,a_{r_m},y_1^{\tau_j}, \hdots y_{s_m}^{\tau_j})|^2 \Big) \Big| < \frac{\e^2}{2}.
\end{equation}
We distinguish three cases.
First, assume that $h_m(a_1, \hdots, a_{r_m}, 0, \hdots, 0) = 0$ and that $h_m(1_{A^{\omega,\mathrm{b}}}, \dots, 1_{A^{\omega,\mathrm{b}}},z_1, \dots, z_{s_m})$ is an ordinary $*$-polynomial.
In this case 
\[
|h_m(a_1, \hdots, a_{r_m}, y_1, \hdots, y_{s_m})|^2 = \sum_{j=1}^k p_j|h_m(a_1,\hdots,a_{r_m},y_1^{\tau_j}, \hdots y_{s_m}^{\tau_j})|^2
\]
since $p_1, \hdots, p_k$ are pairwise orthogonal projections.
Second, assume $h_m$ is of the form $z_i - a$ for some $i \in \N$ and some element $a \in \mathrm{C}^*(\{ \alpha^\omega_g(a_i) \mid i \in \N,\ g\in G\})$ with $\|a\|=1$. 
In this case we have
\begingroup
\allowdisplaybreaks
\begin{align*}
& \Big| \tau\big( |h_m(a_1, \hdots, a_{r_m}, y_1, \hdots, y_{s_m})|^2 \big) - \tau\Big( \sum_{j=1}^k p_j|h_m(a_1,\hdots,a_{r_m},y_1^{\tau_j}, \hdots y_{s_m}^{\tau_j})|^2 \Big) \Big| \\
& =  \Big| \tau\Big( \Big| \sum_{j=1}^k p_j y_i^{\tau_j} - a \Big|^2 \Big) - \tau\Big( \sum_{j=1}^k p_j |y_i^{\tau_j} - a|^2 \Big) \Big|\\
&= \Big| \tau(|a|^2) - \tau\Big( \sum_{j=1}^k p_j |a|^2 \Big) \Big|\\
&\leq  \tau\Big( 1 - \sum_{j=1}^k p_j \Big)\\
\overset{\eqref{eq:sum}}&{<} 1-t \overset{\eqref{eq:choose_t}}{<} \frac{\e^2}{2}.
\end{align*}
\endgroup
Third, assume $h_m$ is of the form $g \cdot z_i - z_{i'}$ for some $i,i' \in \N$ and $g \in F$. Then we have
\begingroup
\allowdisplaybreaks
\begin{align}
& \Big| \tau\big( |h_m(a_1, \hdots, a_{r_m}, y_1, \hdots, y_{s_m})|^2 \big) - \tau\Big( \sum_{j=1}^k p_j|h_m(a_1,\hdots,a_{r_m},y_1^{\tau_j}, \hdots y_{s_m}^{\tau_j})|^2 \Big) \Big| \nonumber\\
& =  \Big| \tau\Big( \Big| \alpha^\omega_g\Big( \sum_{j=1}^k p_jy_i^{\tau_j} \Big) - \sum_{j'=1}^k p_{j'} y_{i'}^{\tau_{j'}} \Big|^2 \Big) - \tau\Big( \sum_{j=1}^k p_j |\alpha^\omega_g(y_i^{\tau_j}) - y_{i'}^{\tau_j}|^2 \Big) \Big|\nonumber\\
&= \Big| \tau\Big( \sum_{j=1}^k (\alpha^\omega_g(p_j) - p_j) \alpha^\omega_g(|y_i^{\tau_j}|^2) \Big) - \tau\Big( \sum_{j=1}^k(\alpha^\omega_g(p_j)-p_j) \alpha^\omega_g(y_i^{\tau_j})^* \sum_{j'=1}^k p_{j'} y_{i'}^{\tau_{j'}} \Big) \nonumber\\
 &- \tau\Big( \Big( \sum_{j'=1}^k p_{j'} y_{i'}^{\tau_{j'}} \Big)^* \sum_{j=1}^k (\alpha^\omega_g(p_j) - p_j)\alpha^\omega_g(y_i^{\tau_j}) \Big) \Big| \nonumber\\
 &\leq \Big\| \sum_{j=1}^k (\alpha^\omega_g(p_j) - p_j) \alpha^\omega_g(|y_i^{\tau_j}|^2) \Big\|_{2,\tau} 
 + 2 \Big\| \sum_{j=1}^k( \alpha^\omega_g(p_j) - p_j) \alpha^\omega_g (y_i^{\tau_j}) \Big\|_{2,\tau} \Big\|\sum_{j=1}^k p_j y_{i'}^{\tau_j} \Big\|_{2,\tau} \nonumber\\
  & \leq \Big\| \sum_{j=1}^k (\alpha^\omega_g(p_j) - p_j) \alpha^\omega_g(|y_i^{\tau_j}|^2) \Big\|_{2,\tau}
  + 2 \Big\| \sum_{j=1}^k( \alpha^\omega_g(p_j) - p_j) \alpha^\omega_g( y_i^{\tau_j}) \Big\|_{2,\tau}.\label{eq:intermediate_bringing_projections_out}
\end{align}
\endgroup
Note that for any $g \in G$ and any positive contractions $c_1, \hdots, c_k \in A^{\omega,\rm b}$ commuting with the $p_j$ and $\alpha^\omega_g(p_j)$ one has that 
\[
\begin{array}{cl}
\multicolumn{2}{l}{ \displaystyle \Big\| \sum_{i=1}^k (\alpha^\omega_g(p_i)-p_i)c_i \Big\|^{2}_{2,\tau} }\\
=& \displaystyle \sum_{i=1}^k \tau\big( (\alpha^\omega_g(p_i)-p_i)^2 c_i^2 \big)
+ \sum_{i,j=1\atop i\neq j}^k \tau\big( c_i (\alpha^\omega_g(p_i) - p_i)(\alpha^\omega_g(p_j)-p_j)c_j \big) \\
\leq& \displaystyle \sum_{i=1}^k \Big(\tau((\alpha^\omega_g(p_i)-p_i)^2) - \sum_{\substack{j=1, \hdots, k\\ j \neq i}}\tau(c_i\alpha^\omega_g(p_i)p_jc_j)-\sum_{\substack{j=1, \hdots, k\\ j \neq i}}\tau(c_ip_i\alpha^\omega_g(p_j)c_j) \Big)\\
\leq& \displaystyle \sum_{i=1}^k \|\alpha^\omega_g(p_i) - p_i\|_{2,\tau}^2,
\end{array}
\]
where in the last inequality we used the tracial property and the fact that the $c_i$ commute with the $p_j$ and $\alpha_g^\omega(p_j)$ to show that the last two terms can be rewritten as minus the trace of positive elements.

In particular we have for $g \in F$ that
\begin{equation}\label{intermediate_2}
\Big\| \sum_{j=1}^k (\alpha^\omega_g(p_j) - p_j)\alpha^\omega_g(|y_i^{\tau_j}|^2) \Big\|_{2,\tau} \leq \sqrt{ \sum_{i=1}^k \|\alpha^\omega_g(p_i) - p_i\|_{2,\tau}^2} \overset{\eqref{eq:invariance}}{<} \delta \overset{\eqref{eq:choose_delta}}{=} \frac{\e^2}{18}.
\end{equation}
For $j=1, \hdots, k$ we have that $\alpha^\omega_g(y_i^{\tau_j})$ is a contraction that can be written as a linear combination of positive contractions commuting with the $p_i$ and $\alpha_g^\omega(p_i)$. An application of the triangle inequality yields 
\begin{equation}\label{intermediate_3}
\Big\| \sum_{i=1}^k (\alpha^\omega_g(p_i)-p_i)\alpha^\omega_g(y_i^{\tau_j}) \Big\|_{2,\tau} \leq 4\sqrt{\sum_{i=1}^k \|\alpha^\omega_g(p_i) - p_i\|_{2,\tau}^2} \overset{\eqref{eq:invariance}}{<} 4\delta \overset{\eqref{eq:choose_delta}}{=} \frac{2\e^2}{9} .
\end{equation}
Combining \eqref{eq:intermediate_bringing_projections_out} with \eqref{intermediate_2} and \eqref{intermediate_3} we get 
\begin{align*}
& |\tau(|h_m(a_1, \hdots, a_{r_m}, y_1, \hdots, y_{s_m})|^2) - \tau(\sum_{j=1}^k p_j|h_m(a_1,\hdots,a_{r_m},y_1^{\tau_j}, \hdots y_{s_m}^{\tau_j})|^2)| \\
&< \frac{9\e^2}{18}= \frac{\e^2}{2}.
\end{align*}
Thus, we have indeed shown that \eqref{eq:bringing_projections_out} holds for all $\tau \in T_\omega(A)$ and $m=1, \hdots, \ell$.
From \eqref{eq:definition_b} we see that
\begin{equation}\label{eq:estimate_h_b}
\sum_{j=1}^k p_j |h_m(a_1, \hdots, a_{r_m}, y_1^{\tau_j}, \hdots, y_{s_m}^{\tau_j})|^2 \leq \sum_{j=1}^k p_jb^{\tau_j} \quad \text{for } m=1, \hdots, \ell.
\end{equation}
As a consequence, for $\tau \in T_\omega(A)$ and $m=1, \hdots, \ell$ we get
\begingroup
\allowdisplaybreaks
\begin{align*}
\|h_m(a_1, \hdots, a_{r_m},y_1, \hdots, y_{s_m})\|_{2,\tau}^2  &= \tau (|h_m(a_1, \hdots, a_{r_m}, y_1, \hdots, y_{s_m})|^2)\\
\overset{\eqref{eq:bringing_projections_out}}&{<} \tau\Big( \sum_{j=1}^k p_j |h_m(a_1, \hdots, a_{r_m}, y_1^{\tau_j}, \hdots, y_{s_m}^{\tau_j})|^2 \Big) + \frac{\e^2}{2}\\
 \overset{\eqref{eq:estimate_h_b}}&{\leq} \sum_{j=1}^k \tau(p_jb^{\tau_j}) + \frac{\e^2}{2}\\
 \overset{\eqref{eq:delta_inequality}}&{\leq} \sum_{j=1}^k \frac{\e^2}{2} \tau(p_j) + \frac{\e^2}{2}\\
&\leq\frac{\e^2}{2} + \frac{\e^2}{2} = \e^2.
\end{align*}
\endgroup
This concludes the proof. 
\end{proof}

The next theorem gives an alternative formulation of the tracial local-to-global principle that is convenient to use in certain applications.
Before we state it, we introduce some notation:

\begin{notation}
Let $A$ be a \cstar-algebra with an action $\alpha:G  \acts A$ of a countable discrete group.
Given a tracial state $\tau \in T(A)$, denote by $\pi_\tau: A \rightarrow \mathcal{B}(H_\tau)$ the corresponding GNS representation.
Then we can define the representation 
\[
\pi_\tau^\alpha: A \rightarrow \mathcal{B}\big( \ell^2(G,H_\tau) \big),\ \pi_\tau^\alpha(x)(\xi)(h) = \pi_\tau(\alpha_h^{-1}(x))\xi(h).
\]
The left regular representation $\lambda: G \rightarrow \mathcal{U}\big(\ell^2(G,H_\tau) \big)$ defined by $(g \cdot \xi)(h) = \xi(g^{-1}h)$ for $\xi \in \ell^2(G,H_\tau)$ and $g,h \in G$, implements the action $\alpha$ on $\pi_\tau^\alpha(A)$, so we get a continuous extension of the action $\alpha: G \acts \pi_\tau^\alpha(A)''$ on the weak closure.

Notice that $\pi_\tau^\alpha(A)'' \subseteq \prod_{g \in G} \pi_{\tau} (A)''$.
The trace $\tau$ on $A$ extends to a faithful normal trace on $\pi_{\tau} (A)''$, and by composition with the natural quotient map $q_g: \prod_{g \in G} \pi_{\tau} (A)'' \rightarrow \pi_{\tau} (A)''$ onto the summand with index $g \in G$ also to a normal trace on $\prod_{g \in G} \pi_{\tau} (A)''$, which we will denote by $\tilde{\tau}_g$.
Notice that $\tilde{\tau}_g \circ \pi_\tau^\alpha = \tau \circ \alpha_g^{-1}$. Let $(c_g)_{g \in G}$ be a sequence in $(0,1)$ such that $\sum_{g \in G} c_g = 1$.
Then $\tilde{\tau} := \sum_{g \in G} c_g \tilde{\tau}_g$ defines a faithful normal tracial state on $\prod_{g \in G} \pi_{\tau} (A)''$ and hence also on the subalgebra $\pi_\tau^\alpha(A)''$.
In this way we can form the tracial von Neumann algebra ultrapower $(\pi_\tau^{\alpha}(A)'')^\omega$. Note that on bounded subsets of $\prod_{g \in G} \pi_\tau (A)''$, the strong operator topology is induced by the norm $\|\cdot\|_{2,\tilde{\tau}}$, or equivalently, by the seminorms $\{\|\cdot\|_{2,\tilde{\tau}_g} \mid g \in G\}$.
Since $\pi_\tau^\alpha(A)''$ is a von Neumann subalgebra, it follows that on bounded subsets its strong operator topology is also induced by (the restrictions of) these (semi)norms. 
\end{notation}

\begin{remark}\label{remark:tracial_representation}
With the above notation and terminology, the condition in Theorem \ref{theorem:local-to-global} that requires for every $\e >0,\, \ell \in \N$ and $\tau \in \overline{T_\omega(A)}^{w^*}$ the existence of contractions $(y_i)_{i \in \N}$ in $A^{\omega,\mathrm{b}}$ such that 
\[
\|h_m(a_1, \hdots, a_{r_m}, y_1, \hdots, y_{s_m})\|_{2,\tau} < \e \quad \text{for } m=1, \hdots, \ell
\]
is equivalent to the following statement (by Kaplansky's density theorem): 
For every  $\e >0,\, \ell \in \N$ and $\tau \in \overline{T_\omega(A)}^{w^*}$, there exist contractions $(y'_i)_{i \in \N}$ in $\pi_\tau^{\alpha^\omega}(A^{\omega,\mathrm{b}})''$ such that \[
\|h_m(a_1, \hdots, a_{r_m}, y'_1, \hdots, y'_{s_m})\|_{2,\tilde{\tau}} < \e \quad \text{for } m=1, \hdots, \ell
.\]
 Making use of the tracial von Neumann algebra ultrapowers, this is also equivalent to the following statement:
For every $\tau \in \overline{T_\omega(A)}^{w^*}$ there are contractions $(y''_i)_{i \in \N}$ in $(\pi_\tau^{\alpha^\omega}(A^{\omega,\mathrm{b}})'')^\kappa$ such that $h_m(a_1, \hdots, a_{r_m},y''_1,\hdots,y''_s) = 0$ for every $m \in \N$. 

\end{remark}
\begin{theorem} \label{theorem:reformulation_local_to_global}
Let $A$ be a $\sigma$-unital \cstar-algebra with $T(A)$ non-empty and compact.
Let $\alpha: G \acts A$ be an action by an amenable countable discrete group $G$ and assume it has local equivariant property Gamma w.r.t.\ bounded traces.
Let $\omega$ and $\kappa$ be two free ultrafilters on $\N$.
Let $\delta: G\acts D$ be an action on a separable \cstar-algebra and let $B \subset D$ be a separable, $\delta$-invariant \cstar-subalgebra.
Suppose $\varphi: (B, \delta) \rightarrow (A^{\omega,\rm b}, \alpha^\omega)$ is an equivariant $*$-homomorphism.
Then the following are equivalent:
\begin{enumerate}[topsep=1ex]
\item For every $\tau \in \overline{T_\omega(A)}^{w^*}$, there exists an equivariant $*$-homomorphisms $\varphi^\tau: (D, \delta) \rightarrow ((\pi_\tau^{\alpha^\omega}(A^{\omega,\rm b})'')^{\kappa}, (\alpha^\omega)^{\kappa})$ such that $\varphi^\tau\rvert_B = \pi_\tau^{\alpha^\omega}\circ\varphi$.\label{enum:local_extension}
\item There is an equivariant $*$-homomorphism $\bar{\varphi}: (D, \delta) \rightarrow (A^{\omega,\rm b}, \alpha^\omega)$ with $\bar{\varphi}\big \rvert_B = \varphi$.\label{enum:global_extension}
\end{enumerate}
\end{theorem}
\begin{proof}
It is clear that \ref{enum:global_extension} implies \ref{enum:local_extension}. 
To prove the other implication, take a countable dense $\Q[i]$-$*$-subalgebra $C \subset D$ such that it is $\delta$-invariant and such that $C \cap B$ is also dense in $B$.
By inductively enlarging $C$ we may in addition assume that for each contraction $x \in C$, one has $1-\sqrt{1-x^*x}\in C$.
Let $\mathcal{P}$ denote the countable family of $G$-$*$-polynomials with coefficients in $A^{\omega,\rm b}$ in the variables $\{X_c\}_{c \in C}$ encoding all relations in $C$:
\begin{itemize}[itemsep=1ex,topsep=1ex]
\item $g\cdot X_c - X_{\delta_g(c)}$ for all $c \in C$ and $g \in G$
\item $\lambda X_c + X_{c'} - X_{\lambda c + c'}$ for all $c,c' \in C$ and $\lambda \in \Q[i]$
\item $X_cX_{c'} - X_{cc'}$ for $c,c' \in C$
\item $X_c^*-X_{c^*}$ for $c \in C$
\item $\varphi(b) - X_b$ for $b \in B \cap C.$ 
\end{itemize}
It follows from \ref{enum:local_extension} that for every $\tau \in \overline{T_\omega(A)}^{w^*}$, the equations in $\mathcal{P}$ have exact solutions in $(\pi_\tau^{\alpha^\omega}(A^{\omega,\mathrm{b}})'')^\kappa$. By Remark \ref{remark:tracial_representation} this means precisely that all conditions to apply Theorem \ref{theorem:local-to-global} are fulfilled, and we can find exact solutions to all equations in $\mathcal{P}$ in $A^{\omega,\rm b}$.
This is equivalent to the existence of a $\Q[i]$-linear, $*$-preserving, multiplicative, equivariant map $\bar{\varphi}: C \rightarrow A^{\omega,\rm b}$ with $\bar{\varphi}\rvert_{ B\cap C}= \varphi\rvert_{B\cap C}$. 
We observe that $\bar{\varphi}$ is contractive.
Indeed, if $x \in C$ is a contraction, then $y =1-\sqrt{1 -x^*x}$ is a self-adjoint element also belonging to $C$, which satisfies
\[
x^*x+y^2-2y = x^*x +(y-1)^2 - 1 = 0.
\]
Hence 
\[
\bar{\varphi}(x)^*\bar{\varphi}(x) + \bar{\varphi}(y)^2 - 2\bar{\varphi}(y) = 0, 
\]
or equivalently,
\[
\bar{\varphi}(x)^*\bar{\varphi}(x) + (1-\bar{\varphi}(y))^2 = 1.
\]
We see that $\bar{\varphi}(x)^*\bar{\varphi}(x)$ is a contraction and hence, $\bar{\varphi}(x)$ is as well.
In conclusion, $\varphi$ extends to an equivariant $*$-homomorphism $\bar{\varphi}: (D,\delta) \rightarrow (A^{\omega,\rm b},\alpha^\omega)$ with $\bar{\varphi}_B=\varphi$. 
\end{proof}

In many cases of interest we get the following corollary by Proposition \ref{prop:unital_Gamma_agrees_local_Gamma}, which directly generalizes and recovers the technical machinery related to uniform property Gamma from the non-dynamical setting; see \cite[Lemma 4.1]{CETWW21}.
We note that upon close inspection of our proof so far, this particular corollary can be obtained based on \cite[Lemma 3.6]{CETWW21} without relying on the preprint \cite{CCEGSTW}.

\begin{corollary} \label{corollary:local-to-global-algebraically-simple}
Let $A$ be a separable, simple, nuclear \cstar-algebra with $T(A)$ non-empty and compact, and such that $T^+(A)=\mathbb R^{>0} T(A)$.
Let $\alpha:G \acts A$ be an action by a countable amenable discrete group that has equivariant property Gamma.
Then $\alpha$ satisfies the conclusion of Theorems \ref{theorem:local-to-global} and \ref{theorem:reformulation_local_to_global}.
\end{corollary}

\begin{remark}
For potential subsequent applications of the theory in this article, let us reflect on how we ended up with the main result of this section.
It is worthwhile to note that the amenability of the group $G$ is used (in the proof of Theorem \ref{theorem:local-to-global}) through the F{\o}lner condition exclusively for the purpose to have access to a finite set $H\ssubset G$ that satisfies the conclusion of Lemma \ref{lemma:precursor_1}.
At no other point in the whole chain of argument is it necessary to know that $H$ is actually a set that is almost invariant with respect to $F$, or anything else about $H$ for that matter.
This culminates into the following more explicit observation, which we suspect may be, at some point, interesting to consider for certain actions of non-amenable groups:

Let $A$ be a $\sigma$-unital \cstar-algebra with $T(A)$ non-empty and compact.
Let $\alpha:G \acts A$ be an action of a countable discrete group.
Suppose that for all $\e>0$ and $F\ssubset G$, there exists a finite subset $H\ssubset G$ that satisfies the same conclusion as in Lemma \ref{lemma:precursor_1}.
If $\alpha$ has local equivariant property Gamma w.r.t.\ bounded traces, then $\alpha$ also satisfies the conclusion of Theorems \ref{theorem:local-to-global} and \ref{theorem:reformulation_local_to_global}.
\end{remark}

\section{Equivariant Jiang--Su stability}

In this section we use the dynamical tracial local-to-global principle derived in the previous section combined with von Neumann algebraic results to conclude that for actions of countable amenable groups on separable, simple, nuclear, finite, $\mathcal{Z}$-stable \cstar-algebras, equivariant property Gamma implies equivariant $\mathcal{Z}$-stability.
Although one can get by with known variations of Ocneanu's theorem \cite{Ocneanu85} for many applications treated in this section, our most general results here need a more general McDuff-type theorem for actions of amenable groups on von Neumann algebras, which we import from our recent work \cite{SzaboWouters23md}.

We begin by reducing the problem of equivariant $\mathcal{Z}$-stability to the existence of so-called tracially large c.p.c.\ order zero maps $M_n \rightarrow F_\omega(A)^{\tilde{\alpha}_\omega}$ for $n \geq 2$. The argument is well-known to experts and traces back to the work of Matui--Sato \cite{MatuiSato12}. It makes use of an equivariant version of their property (SI), for which the general framework needed here was developed in \cite{Szabo21si}.

\begin{definition}[{\cite[Definition 2.5]{Szabo21si}, \cite[Definition 1.3]{CastillejosLiSzabo}}]
Let $A$ be a separable, simple \cstar-algebra with $T^+(A)\neq \emptyset$. 
\begin{enumerate}
\item We say that a positive contraction $f \in F_\omega(A)$ is \emph{tracially supported at 1} if the following holds: For every non-zero positive element $a \in \mathcal{P}(A)$, there exists a constant $\kappa = \kappa(f,a) >0$ such that for every $\tau \in \tilde{T}_\omega(A)$ with $0<\tau(a)<\infty$, one has $\inf_{k \in \N} \tau_a(f^k) \geq \kappa \tau(a)$.
\item A positive element $e \in F_\omega(A)$ is called \emph{tracially null}, if $e\in\mathcal{J}_A$ in the sense of Definition \ref{def:AomegacapAprime}.
\item Let $B$ be a unital \cstar-algebra. A c.p.c.\ order zero map $\phi:B \rightarrow F_\omega(A)$ is called \emph{tracially large} if $\tau_a \circ \phi(1) = \tau(a)$ for all non-zero positive elements $a \in \mathcal{P}(A)$ and $\tau \in \tilde{T}_\omega(A)$ with $\tau(a) < \infty$. 
\end{enumerate}
\end{definition}
 \begin{remark}\label{remark:pedersen_check_single_element}
 It follows from \cite[Proposition 2.4]{Szabo21si} that any of the conditions above hold for all non-zero positive elements $a \in \mathcal{P}(A)$ if and only if they hold for just a single such element, so in practice it suffices to check them for a single $a \in \mathcal{P}(A)_+ \setminus\{0\}$.
 \end{remark}

\begin{definition}[{\cite[Definition 2.7]{Szabo21si}}]
Let $A$ be a separable, simple \cstar-algebra with $T^+(A)\neq \emptyset$ and an action $\alpha:G \acts A$ of a countable discrete group. We say that $\alpha$ has \emph{equivariant property (SI)} if the following holds:

Whenever $e,f \in F_\omega(A)^{\tilde{\alpha}_\omega}$ are two positive contractions such that $f$ is tracially supported at 1 and $e$ is tracially null, there exists a contraction $s \in F_\omega(A)^{\tilde{\alpha}_\omega}$ such that $fs=s$ and $s^*s=e$. 
\end{definition}

It follows from \cite[Corollary 4.3]{Szabo21si} that all actions of amenable groups on non-elementary, separable, simple, nuclear \cstar-algebras with strict comparison have property (SI).
Together with the following theorem, it combines into a powerful sufficient criterion for equivariant Jiang--Su stability.
This is not new to the experts, but has never been formally stated in this generality before, so we shall give the proof for the reader's convenience. 

\begin{theorem} \label{theorem:reduction_equivariant_Z-stability}
Let $A$ be a separable, simple \cstar-algebra with $T^+(A) \neq \emptyset$ and $\alpha: G \acts A$ an action of a countable discrete group with equivariant property (SI).
Then $\alpha$ is equivariantly $\mathcal{Z}$-stable if and only if for every $n \in \N$, there exists a unital $*$-homomorphism $M_n \rightarrow (A^\omega\cap A')^{\alpha^\omega}$.
\end{theorem}
\begin{proof}
Since the ``only if'' part can be obtained with the standard argument sketched in Remark \ref{remark:Z-stability_implies_Gamma}, we prove the ``if'' part.
Given $n \in \N$, let $\phi': M_n \rightarrow (A^\omega\cap A')^{\alpha^\omega}$ be a unital $*$-homomorphism.
By Proposition \ref{prop:fixed_point_surjectivity}, we can find a tracially large c.p.c.\ order zero map  $\phi:M_n \rightarrow F_\omega(A)^{\tilde{\alpha}_\omega}$ that lifts $\phi'$.
Set $e:= 1_{F_\omega(A)} - \phi(1)$ and set $f:= \phi(e_{1,1})$.
Both are positive contractions in $F_\omega(A)^{\tilde{\alpha}_\omega}$.
Since $\phi$ is tracially large, we can conclude immediately that $e$ is tracially null.
Since $e_{1,1}$ is a projection, it follows that $\phi(e_{1,1}) - \phi(e_{1,1})^m$ is tracially null for any $m \geq 1$. Moreover, for every $\tau \in \tilde{T}_\omega(A)$ and $a \in \mathcal{P}(A)_+ \setminus\{0\}$ such that $\tau(a) < \infty$, the functional $\tau_a\circ\phi$ is a bounded trace and therefore a multiple of the unique tracial state on $M_n$. So for every $k \in \N$ we have
\[
\tau_a(f^k) =\tau_a(\phi(e_{1,1})^k) = \tau_a(\phi(e_{1,1})) = \frac{1}{n}\tau(a).
\]
This proves that $f$ is tracially supported at 1. 

Since $\alpha$ has equivariant property (SI), we can find a contraction $s \in F_\omega(A)^{\tilde{\alpha}_\omega}$ such that $fs=s$ and $s^*s=e$.
By \cite[Theorem 5.1]{RordamWinter10} this implies the existence of a unital $*$-homomorphism from the dimension drop algebra $Z_{n,n+1}$ into $F_\omega(A)^{\tilde{\alpha}_\omega}$.
As $\mathcal{Z}$ is an inductive limit of those algebras we find a unital $*$-homomorphism $\mathcal{Z} \rightarrow F_\omega(A)^{\tilde{\alpha}_\omega}$\footnote{This is a standard reindexation trick. Alternatively one can deduce this for example by a combination of \cite[Corollary 3.9, Lemma 4.2]{BarlakSzabo16}.}.
This implies equivariant $\mathcal{Z}$-stability by \cite[Corollary 3.8]{Szabo18ssa}.
\end{proof}

We shall now prove that for actions of amenable groups on simple nuclear $\mathcal Z$-stable \cstar-algebras, equivariant uniform property Gamma is equivalent to equivariant $\mathcal Z$-stability.
We end up giving two separate arguments to prove this result in two cases.
Firstly, we prove this result for actions on \cstar-algebras that have a compact non-empty tracial state space and no unbounded traces, for which it is sufficient to appeal to Corollary \ref{corollary:local-to-global-algebraically-simple}.
Secondly, we prove the result in full generality, but this requires the full power of our theory based on the results from \cite{CCEGSTW}.

Let us proceed in the first case:

\begin{theorem} \label{theorem:Gamma-equivalent-to-Z-stable-special-case}
Let $A$ be a separable, nuclear, simple $\mathcal Z$-stable \cstar-algebra with $T(A)$ non-empty and compact, and such that $T^+(A)=\mathbb R^{>0}T(A)$.
Let $\alpha: G \acts A$ be an action of a countable discrete amenable group.
If $\alpha$ has equivariant property Gamma, then $\alpha$ is equivariantly $\mathcal{Z}$-stable.   
\end{theorem}
\begin{proof}
By Theorem \ref{theorem:reduction_equivariant_Z-stability}, it suffices to construct a unital $*$-homomorphism $M_n \rightarrow (A^\omega\cap A')^{\alpha^\omega}$ for $n \geq 2$.
We appeal to Corollary \ref{corollary:local-to-global-algebraically-simple} and hence know that $\alpha$ obeys the conclusion of Theorem \ref{theorem:local-to-global}.

Let $(a_k)_{k \in \N}$ be a dense sequence in $A$.
Then the existence of such a desired $*$-homomorphism is equivalent to the existence of elements $e_{1,1}, e_{2,1}, \hdots, e_{n,1} \in A^\omega$ satisfying the equations
\[
e_{i,1}e_{j,1}=0 ,\, e_{i,1}^*e_{j,1}=\delta_{ij} e_{1,1},\, e_{1,1}^2=e_{1,1}, \, \alpha^\omega_g(e_{i,1})=e_{i,1} \text{ and } a_k e_{i,1}-e_{i,1}a_k=0
\]
for all $g\in G$, $i,j=2, \hdots, n$ and $k \in \N$.
By Theorem \ref{theorem:local-to-global}, it suffices to show that for each $\varepsilon>0$, finite subset $F\subset\subset G$, $m\in\N$ and every tracial state $\tau\in T(A)$, there exists contractions $f_{1,1}, f_{2,1}, \hdots, f_{n,1} \in A$ satisfying
\[
\|f_{i,1}f_{j,1}\|_{2,\tau}<\varepsilon,\, \|f_{i,1}^*f_{j,1} - \delta_{ij} f_{1,1}\|_{2,\tau}<\varepsilon,\ \|f_{1,1}^2-f_{1,1}\|_{2,\tau} < \e,\]
\[ \| \alpha_g(f_{i,1})-f_{i,1}\|_{2,\tau}<\varepsilon \text{ and } \|a_k f_{i,1}- f_{i,1}a_k\|_{2,\tau}<\varepsilon
\]
for all $g\in F$, $i,j =2, \hdots, n$ and $1\leq k \leq m$.
This is the case, however, if and only if for every $\tau\in T(A)$ there exists a unital equivariant $*$-homomorphism $M_n\to \big( (\pi_\tau^\alpha(A)'')^\omega\cap A' \big)^{\alpha^\omega}$.
Since $A$ is nuclear, the tracial von Neumann algebra $N_\tau:=\pi_\tau^\alpha(A)''$ is injective, see for example \cite[Theorem IV.2.2.13]{Blackadar}.
Since it does not have any direct summand of type I, it follows from Connes' theorem \cite{Connes76} that $N_\tau\bar{\otimes}\mathcal R\cong N_\tau$.
Hence the claim follows directly from \cite[Theorem A]{SzaboWouters23md}.
\end{proof}

Next we proceed in the second and more general case:

\begin{lemma} \label{lemma:equivalent_absorption_criteria}
Let $A$ be a $\sigma$-unital \cstar-algebra with $T(A)$ non-empty and compact. Let $\alpha: G \acts A$ be an action by a countable discrete amenable group $G$ and assume it has local equivariant property Gamma w.r.t.\ bounded traces. Then the following are equivalent:
\begin{enumerate}
\item For all $\tau \in T(A)$, there exists a unital $*$-homomorphisms $\varphi^\tau: M_n \rightarrow ((\pi_\tau^\alpha(A)'')^\omega)^{\alpha^\omega}$.\label{enum:local_inclusion_M2}
\item There exists a unital $*$-homomorphism $\varphi: M_n \rightarrow (A^{\omega,\rm b})^{\alpha^\omega}$.\label{enum:global_inclusion_M2} 
\end{enumerate}
\end{lemma}
\begin{proof}
For each $\tau \in T(A)$, the map $\pi_\tau^\alpha$ induces a unital $*$-homomorphism 
\[
(A^{\omega,\rm b})^{\alpha^\omega} \rightarrow ((\pi_\tau^\alpha(A)'')^\omega)^{\alpha^\omega},
\] so \ref{enum:global_inclusion_M2} implies \ref{enum:local_inclusion_M2}. We use the fact that $\alpha$ has local equivariant property Gamma to prove the other implication. By Theorem \ref{theorem:reformulation_local_to_global} the following are equivalent:
\begin{enumerate}[label=(\alph*)]
\item For all $\tau \in \overbar{T_\omega(A)}^{w^*}$, there exists a unital equivariant $*$-homomorphism $\varphi^\tau: (M_n, \mathrm{id}_{M_n}) \rightarrow ((\pi_\tau^{\alpha^\omega}(A^{\omega,\rm b})'')^\kappa, (\alpha^\omega)^\kappa)$.\label{enum:local_inclusion_tensor}
\item There exists a unital equivariant $*$-homomorphism $\varphi: (M_n, \mathrm{id}_{M_n}) \rightarrow (A^{\omega,\rm b}, \alpha^\omega)$.\label{enum:global_inclusion_tensor}
\end{enumerate} 
Statement \ref{enum:global_inclusion_tensor} is equivalent to statement \ref{enum:global_inclusion_M2} above.
Hence, in order to prove the implication it suffices to prove that statement \ref{enum:local_inclusion_M2} implies \ref{enum:local_inclusion_tensor}. Take $\tau \in \overbar{T_\omega(A)}^{w^*}$ and denote its restriction to $A$ by $\sigma$.
The canonical map $A \rightarrow A^{\omega,\rm b}$ induces an equivariant $*$-homomorphism $(\pi_{\sigma}^\alpha(A),\alpha) \rightarrow (\pi_\tau^{\alpha^\omega}(A^{\omega,\rm b})'',\alpha^\omega)$ that is continuous on the unit ball with respect to the strong operator topology.
Hence, it can be extended to a unital equivariant $*$-homomorphism $\pi_{\sigma}^\alpha(A)'' \rightarrow \pi_\tau^{\alpha^\omega}(A^{\omega,\rm b})''$.
Combining this with \ref{enum:local_inclusion_M2}, this means we can find a unital $*$-homomorphism $M_n \rightarrow ((\pi_\tau^{\alpha^\omega}(A^{\omega,\rm b})'')^\kappa)^{(\alpha^\omega)^\kappa}$. This ends the proof.
\end{proof}

\begin{theorem} \label{theorem:Gamma-equivalent-to-Z-stable}
Let $A$ be a separable, simple, nuclear $\mathcal Z$-stable \cstar-algebra such that $T^+(A) \neq \emptyset$.
Let $\alpha: G \acts A$ be an action of a countable discrete amenable group.
If $\alpha$ has equivariant property Gamma, then $\alpha$ is equivariantly $\mathcal{Z}$-stable.   
\end{theorem}
\begin{proof}
Combining \cite[Corollary 4.3]{Szabo21si} with Theorem \ref{theorem:reduction_equivariant_Z-stability}, we see that given $n\geq 2$, it suffices to construct a unital $*$-homomorphism $M_n \rightarrow (A^\omega\cap A')^{\alpha^\omega}$. 
For convenience, let us specify a (possibly different) free ultrafilter $\kappa$ on $\N$. 
It suffices to show that we can construct a unital $*$-homomorphism $M_n \rightarrow \big((A^\omega \cap A')^{\kappa, \mathrm{b}}\big)^{(\alpha^\omega)^\kappa}$, as a reindexation trick will then yield the required unital $*$-homomorphism $M_n \rightarrow (A^\omega\cap A')^{\alpha^\omega}$. 
By Theorem \ref{theorem:gamma_equivalence}, we conclude that $\alpha^\omega= G \acts A^\omega \cap A'$ has local equivariant property Gamma w.r.t.\ bounded traces. 
Thus, by Lemma \ref{lemma:equivalent_absorption_criteria} it suffices to prove that for all $\tau \in T(A^\omega\cap A')$ there exists a unital $*$-homomorphism $\varphi^\tau: M_n \rightarrow \big( (\pi_\tau^{\alpha^\omega}(A^\omega\cap A')'')^{\kappa} \big)^{(\alpha^\omega)^\kappa}$.
We note that the \cstar-algebra $A$ has uniform property Gamma.
Using the same trick as in the proof of Proposition \ref{prop:tracesF(A)}, we conclude that $A^\omega\cap A'\cong B^\omega\cap B'$ for a hereditary subalgebra $B\subset A\otimes\mathbb K$ such that $T^+(B)=\mathbb{R}^{>0} T(B)$ and $T(B)$ is compact.
Using \cite[Theorem 4.6]{CETW22}, we can hence conclude that there exists a unital $*$-homomorphism $M_2\to A^\omega\cap A'$.
By a standard reindexation trick, we can argue that such a $*$-homomorphism can be chosen to additionally commute with any specified separable subset of $A^\omega\cap A'$.

Fix $\tau \in T(A^\omega\cap A')$.
We show first that $N_\tau:= \pi_\tau^{\alpha^\omega}(A^\omega\cap A')''$ contains a $\|\cdot\|_{2,\tau}$-separable, $\alpha^\omega$-invariant von Neumann subalgebra that tensorially absorbs the hyperfinite II$_1$-factor.
By the aforementioned property of $A$, we can find a unital embedding $\phi_1: M_2 \rightarrow \pi_\tau^{\alpha^\omega}(A^\omega\cap A') \subseteq N_\tau$. 
Set
\[
B_1 := C^*\Big( \bigcup_{g\in G} \alpha^\omega_g(\phi_1(M_2)) \Big).
\]
Using that $B_1$ is a separable subquotient of $A^\omega\cap A'$, we can use the aforementioned property of $A$ again and and find a unital embedding 
\[
\phi_2: M_2 \rightarrow \pi_\tau^{\alpha^\omega}(A_\omega\cap A')\cap B_1' \subseteq N_\tau\cap B_1'.
\]
Set
\[
B_2 := C^*\Big( B_1\cup \bigcup_{g\in G} \alpha^\omega_g(\phi_2(M_2)) \Big).
\]
Carry on with this procedure inductively, i.e., given the \cstar-algebra $B_i\subset \pi_\tau^{\alpha^\omega}(A^\omega\cap A')$, find a unital $*$-homomorphism $\phi_{i+1}: M_2 \rightarrow \pi_\tau^{\alpha^\omega}(A^\omega\cap A')\cap B_i'$ and set
\[
B_{i+1} := C^*\Big( B_i \cup \bigcup_{g\in G} \alpha^\omega_g(\phi_{i+1}(M_2)) \Big).
\]
Define $\mathcal{B}:=\overline{\cup_{i \in \N} B_i}^{\|\cdot\|_{2,\tau}}\subset N_\tau.$
Then $\mathcal{B}$ is a $\|\cdot\|_{2,\tau}$-separable, $\alpha^\omega$-invariant von Neumann subalgebra of $N_\tau$ such that additionally $\mathcal{B} \cong \mathcal{B} \bar{\otimes} \mathcal{R}$ by \cite[Corollary 3.8]{SzaboWouters23md} because it satisfies the McDuff-type criterion (existence of a unital $*$-homomorphism $\mathcal{R} \rightarrow \mathcal{B}_\omega$)
by construction.
Denote the restriction of $\alpha^\omega$ to $\mathcal{B}$ by $\beta$.
By \cite[Theorem A]{SzaboWouters23md}, it follows that $\beta$ is cocycle conjugate to $\beta\otimes \mathrm{id}_\mathcal{R}$. 
In particular we can find a unital $*$-homomorphism 
\[
M_n \rightarrow (\mathcal{B}^{\kappa})^{\beta^\kappa} \subset \big( (\pi_\tau^{\alpha^\omega}(A^\omega\cap A')'')^{\kappa} \big)^{(\alpha^\omega)^\kappa}.
\]
This ends the proof.
\end{proof}

\end{document}